\documentclass{siamart220329} %

\usepackage{amsmath}
\usepackage{amssymb}
\usepackage{mathtools}
\usepackage{float}
\usepackage{hyperref,cleveref}
\usepackage{tikz}
\usepackage[most]{tcolorbox}
\usetikzlibrary{plotmarks}
\usepackage{pgfplots}
\usepackage{algorithm}
\usepackage{algpseudocode}
\usepackage{multirow}
\usepackage{hhline}
\usepackage{multicol}
\usepackage{wasysym}
\usepackage{enumitem}
\usepackage{xspace}
\usepackage{subcaption}

\newcommand{\N}{\mathbb{N}}
\newcommand{\R}{\mathbb{R}}
\newcommand{\bigO}{\mathcal{O}}
\newcommand{\bx}{\mathbf{x}}
\newcommand{\hatbx}{\hat{\bx}}
\newcommand{\bp}{\mathbf{p}}
\newcommand{\hatbp}{\hat{\bp}}
\newcommand{\vmin}{v_{\rm{min}}}
\newcommand{\vmax}{v_{\rm{max}}}

\newcommand{\grad}{\nabla}

\newcommand{\mesh}{\mathcal{T}_h}

\newcommand{\mrhyde}{{MrHyDE}}
\newcommand{\dealii}{{deal.II}}
\newcommand{\intrepid}{{Intrepid}}
\newcommand{\mfem}{{MFEM}}

\newtheorem{remark}{Remark}

\newcommand{\REMOVE}[1]{}

\newcommand{\slugmaster}{%
\slugger{sisc}{xxxx}{xx}{x}{x--x}}

\begin{document}
\setcounter{tocdepth}{1}

\title{R-Adaptive Mesh Optimization to Enhance Finite Element Basis Compression}
\author{
  Graham Harper\thanks{Sandia National Laboratories, Computational Mathematics Department,({gbharpe@sandia.gov})}
\and
  Denis Ridzal\thanks{Sandia National Laboratories, Optimization and Uncertainty Quantification Department, ({dridzal@sandia.gov})}
\and
  Tim Wildey\thanks{Sandia National Laboratories, Computational Mathematics Department, ({tmwilde@sandia.gov}).  This paper describes objective technical results and analysis. Any subjective views or opinions that might be expressed in the paper do not necessarily represent the views of the U.S. Department of Energy or the United States Government. Sandia National Laboratories is a multimission laboratory managed and operated by National Technology and Engineering Solutions of Sandia, LLC., a wholly owned subsidiary of Honeywell International, Inc., for the U.S. Department of Energy's National Nuclear Security Administration under contract DE-NA-0003525.}
}

\maketitle

\begin{abstract}
Modern computing systems are capable of exascale calculations, which are revolutionizing the development and application of high-fidelity numerical models in computational science and engineering. 
While these systems continue to grow in processing power, the available system memory has not increased commensurately,
and electrical power consumption continues to grow. 
A predominant approach to limit the memory usage in large-scale applications is to exploit the abundant processing power and continually recompute many low-level simulation quantities, rather than storing them.
However, this approach can adversely impact the throughput of the simulation and diminish the benefits of modern computing architectures.
We present two novel contributions to reduce the memory burden while maintaining performance in simulations based on finite element discretizations.
The first contribution develops dictionary-based data compression schemes that detect and exploit the structure of the discretization, due to redundancies across the finite element mesh.
These schemes are shown to reduce memory requirements by more than 99\% on meshes with large numbers of nearly identical mesh cells.
For applications where this structure does not exist, our second contribution leverages a recently developed augmented Lagrangian sequential quadratic programming algorithm to enable r-adaptive mesh optimization, with the goal of enhancing redundancies in the mesh.
Numerical results demonstrate the effectiveness of the proposed methods to detect, exploit and enhance mesh structure on examples inspired by large-scale applications.

\end{abstract}

\section{Introduction}
\label{sec:introduction}

Modern computing technology has embraced heterogeneous computational architectures with processing units that enable massive concurrency for scientific computing applications.
While these high-performance computing systems, led by Frontier at Oak Ridge National Lab and Aurora at Argonne National Lab, are demonstrating exascale calculations~\cite{top500_6.24}, many applications are limited in fully exploiting the technology due to significant reductions in the amount of system memory per compute thread.
Efforts to compress data output from scientific applications, through dimension reduction and machine learning, do not reduce the amount of system memory required to run the simulations.
New approaches are required to enable higher-fidelity simulations for predictive science while effectively utilizing leadership-class computing systems.
Many scientific applications involving the numerical solution of partial differential equations (PDEs) using finite element methods (FEMs) are limited in their prediction fidelity by system memory rather than processing power.
The shift toward library-based and modular high-performance computing (HPC) applications, through FEM software libraries like \mfem{}~\cite{anderson2021mfem}, \dealii{}~\cite{bangerth2007deal} and \intrepid{}~\cite{bochev2012solving},
presents a unique and far-reaching opportunity for the community to tackle the memory challenge at its source, namely, the data structures used for finite element computations, with few if any changes to the implementations of the PDEs that underlie HPC applications.
Several efforts have been initiated, and in some cases renewed, to address the computational performance challenges associated with the storage of finite element data structures.
Matrix-free approaches, such as Jacobian-free Newton Krylov methods~\cite{knoll2004jacobian} and partial assembly~\cite{anderson2021mfem}, seek to avoid the memory required to store the Jacobians of PDE residuals and utilize only matrix-vector products for linear algebra operations.
By reducing the data associated with linear algebra operations, larger problems can be solved, but constructing effective matrix-free preconditioners is challenging and remains an active area of research~\cite{may2015scalable,clevenger2021comparison}.
Other data reduction techniques such as sum factorizations~\cite{ainsworth2011bernstein,bressan2019sum,mora2019sumfactorization}
seek to exploit structure to reduce the cost in storing or computing the finite element basis functions, and are most beneficial to three-dimensional problems with high-order function approximations.
Elsewhere, data reduction techniques have been used to process vast amounts of data generated by PDE simulations, thereby enabling, e.g., PDE-constrained optimization~\cite{griewank1992achieving,alshehri2024inexact,muthukumar2021randomized} and adjoint-based \textit{a-posteriori} error estimation~\cite{cyr2015towards} for evolutionary PDEs, where the goal is to reduce the cost due to storing the full state trajectory for adjoint simulations.

Complementary to the aforementioned data reduction techniques, in this paper we propose to directly exploit redundancies in the \emph{computational mesh}.
Additionally, we propose to \emph{enhance} mesh redundancies, through a movement of mesh nodes without changes to the mesh topology, also known as adaptive mesh redistribution, r-adaptivity or moving mesh methods~\cite{tang2005moving,budd2009adaptivity,browne2014fast}.
An overarching goal is to enable clustering of mesh cells into a relatively small number of classes based on shape identity or, more broadly, shape similarity.
Some related ideas have been explored in the context of mesh generation and mesh refinement~\cite{suarez2021computing,trujillo2024finite}, i.e., without the challenge of preserving the topology of a \emph{given} mesh.

The cost of storing finite element quantities, such as the evaluations of finite element basis functions or the Jacobians of mesh cell transformations, scales linearly with the number of elements in the mesh, $N$.
Our contributions aim to reduce the $\bigO(N)$ storage cost to $\bigO(m)$, where $m \ll N$, based on the redundancies in the shapes of the mesh cells.
The first contribution is a dictionary-based data compression scheme for problems where a given mesh contains this redundant structure.  
The scheme seeks to compress the evaluations of finite element basis functions at quadrature points, which are linked directly to the geometries of the mesh cells.
The dictionary achieves reductions in storage of more than 99\% for several types of meshes with redundant structure, enabling billion-element simulations in under 100~gigabytes of memory.
The second contribution is a new r-adaptive mesh optimization scheme to enhance the mesh redundancy and, therefore, compression of finite element quantities.
Our scheme combines a novel adaptation of $k$-medoids clustering~\cite{park2009simple} with a sparsity-promoting optimization formulation geared toward lossless compression.
The latter is inspired by the constrained optimization alternative~\cite{ridzal2016meshcorrection} to r-adaptive schemes based on solving the Monge-Amp\`{e}re equation~\cite{cossette2014monge,browne2014fast}.
The proposed inequality constraints on the cell volumes, which prevent tangling and ensure relatively minor deviations from the original mesh, are handled using a fast and scalable augmented Lagrangian method, \cite{antil2023alesqp}.
The overall scheme yields excellent lossless compression for several challenging meshes and, owing to~\cite{antil2023alesqp}, exhibits scalable performance as the mesh size increases.

The remainder of this paper is organized as follows.
In Section~\ref{sec:FEM}, we provide some background on finite element methods, sufficient to motivate in-situ data compression.
In Section~\ref{sec:compression}, we present our dictionary-based compression scheme, and apply it to finite element meshes and discretizations with readily exploitable redundancy.
In Section~\ref{sec:ProblemFormulation}, we introduce our r-adaptive mesh optimization formulation, to enhance redundancy for more general meshes.
The algorithms, including our adaptation of $k$-medoids clustering and the scalable optimization method, are described in Section~\ref{sec:Algorithms}.
We present numerical results in Section~\ref{sec:Results} and provide concluding remarks in Section~\ref{sec:Conclusion}.

\section{Background on finite element methods}
\label{sec:FEM}

A key computational kernel in finite element methods is given by the integrals of the form
\begin{align}
  \label{eq:productintegral}
  & \int_T k(\mathbf{x})
  D^{\alpha_1} \phi_i(\bx)
  D^{\alpha_2} \phi_j(\bx)
  \ d\mathbf{x},
\end{align}
where $T$ is a domain that corresponds to a mesh cell
in a mesh $\mathcal{T}_h$ of the problem domain $\Omega$,
$k(\bx)$ is some (not necessarily scalar) problem coefficient,
$\alpha_1$ and $\alpha_2$ are alpha-indices for the derivative operator $D$,
and $\phi_i(\bx)$ and $\phi_j(\bx)$ are finite element basis functions.
The most common approach to evaluate these integrals 
defines a quadrature and a set of basis functions
on a reference domain $\hat{T}$, which is typically some subset of points
that lie in the hypercube $[-1,1]^d$ or $[0,1]^d$.
$\hat{T}$ is then related to $T$ by the reference to physical mapping $\Phi_T:\hat{T}\to T$,
as shown in Figure \ref{fig: reference to physical},
where $\bx$ refers to physical coordinates on $T$
and $\hatbx$ refers to reference coordinates on~$\hat{T}$.
For a function $f:T\to\R$, we will use the hat notation
$\hat{f}:\hat{T}\to\R$ to define the pullback
$\hat{f}(\hatbx) := f(\Phi_T(\hatbx))$.
Conversely, one defines $f$ as the pushforward of~$\hat{f}$
by~$f(\bx) = \hat{f}(\Phi_T^{-1}(\bx))$.
We may omit the use of the $T$ subscript
when the domain is understood by context.
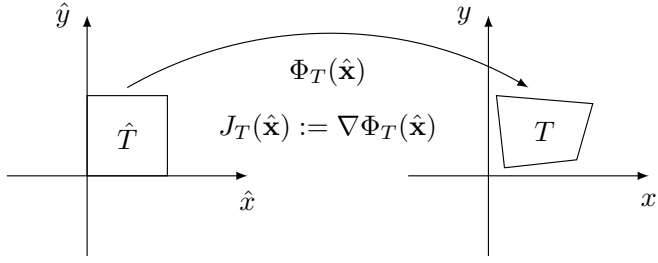
\begin{figure}
  \centering
  \resizebox{!}{3.5cm}{
  \begin{tikzpicture}
  \begin{scope}
  \draw[-latex] (-1,0) -- (2,0); %
  \node at (2,-0.3) {$\hat{x}$}; %
  \draw[-latex] (0,-1) -- (0,2); %
  \node at (-0.3,2) {$\hat{y}$}; %
  \draw (0,0) -- (1,0) -- (1,1) -- (0,1) -- (0,0);
  \node at (0.5,0.5) {$\hat{T}$};
  \end{scope}
  \begin{scope}[xshift=5cm]
  \draw[-latex] (-1,0) -- (2,0); %
  \node at (2,-0.3) {$x$}; %
  \draw[-latex] (0,-1) -- (0,2); %
  \node at (-0.3,2) {$y$}; %
  \draw (0.2,0.1) -- (1.1,0.2) -- (1.3,0.9) -- (0.1,1.0) -- (0.2,0.1);
  \node at (0.7,0.55) {$T$};
  \end{scope}
  \draw[-latex] (0.5,1.1) arc (120:60:5);
  \node at (3,1.3) {$\Phi_T(\hatbx)$};
  \node at (3,0.6) {$J_T(\hatbx):=\nabla\Phi_T(\hatbx)$};
  \end{tikzpicture}}
\caption{Illustration of the reference to physical mapping $\Phi_T$ for a mesh cell $T$ and reference cell $\hat{T}$.}
\label{fig: reference to physical}
\end{figure}

This approach of calculating finite element basis data is the standard approach for many high-performance FEM libraries, such as \mfem{}~\cite{anderson2021mfem}, \dealii{}~\cite{bangerth2007deal} or \intrepid{}~\cite{bochev2012solving}, as it requires only one set of quadrature points and basis evaluations on $\hat{T}$, instead of new quadrature points and bases on each $T$.
However, as this approach requires a change of coordinates,
integrals must be handled more carefully.
For a single scalar-valued function $f(\bx)$,
its integral is computed by the change of coordinates formula
(c.f. \S 2.1.3 in \cite{boffi2013mixed})
\begin{align*}
  \int_T f(\bx) \ d\bx
  &=
  \int_{\hat{T}} \hat{f}(\hatbx) \mbox{det}J_T(\hatbx) \ d\hatbx,
\end{align*}
where $J_T(\hatbx) := \grad \Phi_T(\hatbx)$ is the tensor-valued Jacobian
of the reference to physical mapping.
Integration of $\grad f(\bx)$ involves transforming a vector field to the reference domain,
which requires the application of the covariant transformation (c.f. \S 2.1.3,2.1.4 in \cite{boffi2013mixed})
\begin{align*}
  \int_T \grad f(\bx) \ d\bx
  &=
  \int_{\hat{T}} J^{-\mathsf{T}}_T(\hatbx) \grad \hat{f}(\hatbx) \mbox{det}J_T(\hatbx) \ d\hatbx.
\end{align*}
We will refer to the quantity $J^{-\mathsf{T}}_T \grad \hat{f}(\hatbx)$ on $\hat{T}$ as the
\textit{physical gradient} of $f$ on $T$.
This is easier to understand from the perspective of differential geometry;
here we provide only the necessary information and refer the interested reader to \cite{boffi2013mixed,mora2019sumfactorization}.
The two changes of coordinates presented above are special cases of transformations required to evaluate finite element basis functions, depending on the choice of the discrete space within the so-called de Rham complex.
In $\R^3$, the de Rham complex with the associated spaces and transforms is
\begin{align}
\begin{array}{ccccccc}
H^1(\Omega)
&\xrightarrow{\nabla}& H(\mbox{curl}; \Omega)
&\xrightarrow{\nabla\times}& H(\mbox{div}; \Omega)
&\xrightarrow{\nabla\cdot}& L^2(\Omega)
\\
I(\hatbx)
&&J^{-\mathsf{T}}(\hatbx)
&&\mbox{det}(J(\hatbx))^{-1} J(\hatbx)
&&\mbox{det}(J(\hatbx))^{-1},
\end{array}
\end{align}
where the first row presents the formal Sobolev space names, and the second row presents the pullback transformations needed for function evaluations.
Relating this back to the example shows for $f \in H^1(\Omega)$
the necessary pullback is the identity,
which agrees with the change of coordinates formula for scalar-valued functions.
However, $\grad f \in H(\mbox{curl}; \Omega)$,
meaning the corresponding pullback is $J^{-\mathsf{T}}(\bx)$.
Finally, coming back to~\eqref{eq:productintegral}, computing a so-called  stiffness matrix requires pullbacks of both the test and trial functions,
\begin{align}
  \int_T \grad \phi_i(\bx) \cdot \grad \phi_j(\bx) \ d\bx
  &=
  \int_{\hat{T}} (J^{-\mathsf{T}}_T(\hatbx) \grad \hat{\phi}_i(\hatbx)) \cdot (J^{-\mathsf{T}}_T(\hatbx) \grad \hat{\phi}_j(\hatbx))
  \mbox{det}J_T(\hatbx) \ d\hatbx.
\end{align}
As each of these transforms depend on $J_T(\mathbf{x})$,
two cells $T_1$ and $T_2$ with equal reference to physical
Jacobians $J_{T_1}$ and $J_{T_2}$ will have equal pullbacks
for finite element basis functions.
This observation is valuable in developing a dictionary-based compression scheme for redundant finite element data.

\subsection{Trade-offs between computational complexity and storage cost}
Before presenting the compression scheme, we additionally motivate the dictionary approach by considering the computational complexity and storage cost for finite element basis evaluations.
We briefly discuss alternatives, while focusing on a $d$-dimensional quadrilateral or hexahedral mesh, $d=2,3$, with $N$ cells, a degree-$p$ polynomial basis, and $q$ quadrature points in each dimension.

The alternative to storing basis data with respect to a reference domain
requires storing physical basis functions and gradients on each cell.
Full storage of basis data on the physical cell has the cost
\begin{align}
  \label{eq:storagefull}
  \bigO(N (p+1)^d q^d d).
\end{align}
For simplicial geometry, it is straightforward to define
FEM basis functions on the physical cell by utilizing barycentric coordinates.
This is more difficult for general non-affine quadrilateral and hexagonal geometries,
as the pullback map is nonlinear.
In the simplest case, such basis evaluations cost $\bigO((p+1)d^2)$
floating point operations (flops),
resulting in $\bigO(N (p+1)^{d+1} q^d d^3)$ computational complexity.

However, if only one set of basis data is stored on the reference element
and the transformation Jacobians are stored as well, the storage cost is
\begin{align}
  \bigO((p+1)^d q^d d + N q^d d^2),
\end{align}
where the first term corresponding to reference data is negligible
for any large-scale mesh.
This storage cost is approximately the same as the previous for $p=1$,
but it provides significant savings for $p>1$.
To study the computational complexity of this approach, we note that each reference basis function evaluation costs $\bigO((p+1)d^2)$ flops, i.e., the total reference basis evaluation costs $\bigO((p+1)^{d+1}q^d d^3)$ flops.
However, to transform these reference values to the physical space using the covariant or Piola transforms costs $\bigO(d(2d-1))$ flops for each matrix-vector multiplication.
This matrix-vector multiplication must be performed for every cell, basis function, and quadrature point,
bringing the final computational cost to $\bigO(N (p+1)^d q^d d^2)$ flops.
As the most expensive storage cost here is in the Jacobians,
they could also be evaluated on the fly, but doing so adds
another $\bigO(N d^2 q^d)$ flops.

Finally, one could compute basis data on the reference cell
and then store the basis functions with the appropriate
transformations applied.
The full storage of physical values with respect to reference cells has
\begin{align}
  \bigO(N (p+1)^d q^d d)
\end{align}
cost, which is the same as the first approach.
However, by combining this approach with reference data and
Jacobian-based transformations, the computational complexity is $\bigO(N (p+1)^{d+1} q^d d^2)$ flops instead of $\bigO(N (p+1)^{d+1} q^d d^3)$.
In each of the cases involving storage of data on the reference cell,
there is a tradeoff between storage cost and compute time.
Figure~\ref{fig: partial storage} highlights one example of how such tradeoffs look for a time-dependent simulation on a fixed mesh.
\begin{figure}
  \centering
  \includegraphics[height=0.4\textwidth]{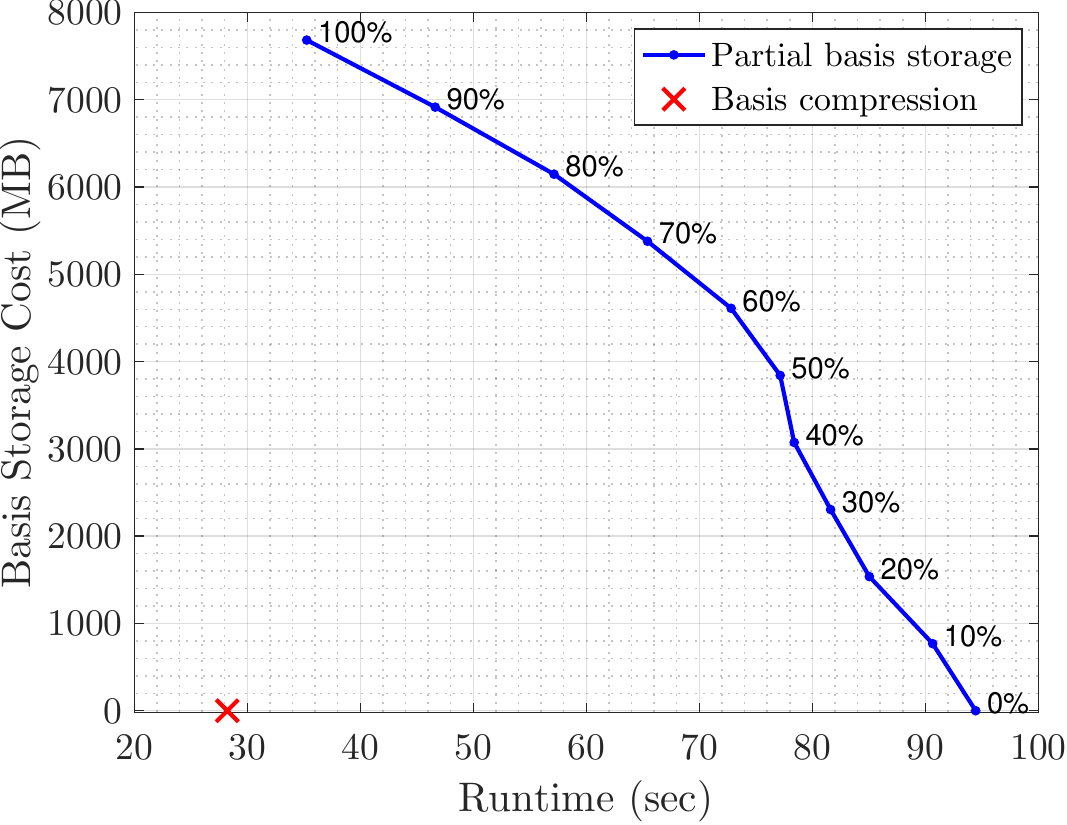}
  \caption{Tradeoffs between storing and recomputing basis data for a heat transfer simulation, where we show the percentage of basis data retained across the mesh.  For the given mesh, containing significant redundancy, the cost of our proposed basis compression scheme is denoted by the $\textcolor{red}{\mathsf{x}}$ mark.}
  \label{fig: partial storage}
\end{figure}
In each of the calculations, with the exception of the 0\% storage scenario and the proposed compression scheme, the number of mesh cells $N$ dominates the storage cost.
At the same time, reducing storage significantly increases runtime.
To alleviate this tradeoff, the proposed compression scheme effectively reduces computations from $N$ mesh cells to $m$ mesh cells, where $m\ll N$, resulting in the storage complexity of
\begin{align}
  \bigO(m (p+1)^d q^d d) ,
\end{align}
while maintaining, and in the case of Figure~\ref{fig: partial storage} improving, runtime performance.

We will explain the mechanics of our compression scheme in the following section.
We note that the reduction of data from $N$ cells to $m$ cells is performed on a cell-by-cell basis.
Therefore, this approach is complementary to other modern data reduction techniques such as sum factorizations.
Typically, sum factorizations reduce the computational complexity by removing the spatial dimension from the exponent, resulting in complexity
\begin{align}
  \bigO(N d^2 (p+1) q).
\end{align}
Additionally exploiting mesh redundancy would reduce the computational complexity of sum factorizations to
\begin{align}
  \bigO(m d^2 (p+1) q).
\end{align}

\section{Exploiting mesh redundancy to compress finite element data}
\label{sec:compression}

This section develops a dictionary-based compression scheme to detect and exploit mesh structure and reduce storage of finite element data.
From the previous section, two mesh cells $T_1$ and $T_2$ with equal Jacobians $J_{T_1} = J_{T_2}$ must have mappings $\Phi_{T_1}$ and $\Phi_{T_2}$ which differ by only a constant.
That is, $T_1$ is a translation of $T_2$ if and only if $J_{T_1} = J_{T_2}$.
In this case, we say the domains $T_1$ and $T_2$ have identical \textit{shapes}\footnote{We discuss cell orientations in Section~\ref{sec:ProblemFormulation}.}.
The following lemma relates cell shapes to finite element operator evaluations.
\begin{lemma}
\label{lem:shapes}
Let $\hat{\phi} \in \mathcal{S}$, with $\mathcal{S} \in \{ H^1(\hat{T}), H(\mathrm{curl};\hat{T}), H(\mathrm{div};\hat{T}),L^2(\hat{T})\}$, be a finite element basis function on the reference domain.
Let $D\in\{1,\mathrm{grad},\mathrm{curl},\mathrm{div}\}$ be an appropriate operator
so that $D\hat{\phi}\in L^2(\hat{T})$.
Then if two mesh cells $T_1$ and $T_2$ have the same shape,
the physical values $D\phi$ are equal and have equal integrals,
\begin{align}
  \int_{T_1} D\phi(\bx) \ d\bx
  &=
  \int_{T_2} D\phi(\bx) \ d\bx.
  \label{eq: basis equality}
\end{align}
\end{lemma}
\begin{proof} The proof follows the formula for the integral in the reverse direction.
Let $w(J)\in\{I,J^{-\mathsf{T}},(\mbox{det}J)^{-1}J,(\mbox{det}J)^{-1}\}$
be the term corresponding to the appropriate transformation for $D\hat{\phi}$, and let $\phi$ denote the corresponding pushforward of $\hat{\phi}$.
Then, if $T_1$ is a translation of $T_2$, we have $J_{T_1} = J_{T_2}$, hence
\[
w(J_{T_1}(\hatbx)) D\hat{\phi}(\hatbx) \mbox{det}J_{T_1}(\hatbx)
=
w(J_{T_2}(\hatbx)) D\hat{\phi}(\hatbx) \mbox{det}J_{T_2}(\hatbx) ,
\]
and therefore
\[
\int_{\hat{T}} w(J_{T_1}(\hatbx)) D \hat{\phi}(\hatbx)
\mbox{det}J_{T_1}(\hatbx) \ d\hatbx
=
\int_{\hat{T}} w(J_{T_2}(\hatbx)) D \hat{\phi}(\hatbx)
\mbox{det}J_{T_2}(\hatbx) \ d\hatbx .
\]
\end{proof}

Lemma~\ref{lem:shapes} is the fundamental motivation for our dictionary scheme.
It implies that finite element operator evaluations on cells with identical shapes are also equal,
\begin{align*}
\int_{T_1} D \phi_i(\mathbf{x}) \cdot D \phi_j(\mathbf{x}) \ d\mathbf{x}
&=
\int_{T_2} D \phi_i(\mathbf{x}) \cdot D \phi_j(\mathbf{x}) \ d\mathbf{x}.
\end{align*}
These observations are neither surprising nor new, and there are FEM libraries, such as \dealii{}, that support a capability to avoid re-computing the reference to physical Jacobian if the mesh cell is a translation of the previous mesh cell.

In a similar yet more sophisticated vein, we propose to compress physical basis quantities by performing a \emph{preprocessing step on the mesh} $\mesh$.
This is accomplished by comparing mesh cells by their Jacobians of the reference to physical transformation, and if two such Jacobians are equal, the physical basis evaluations are saved. %
To understand this better, we present examples of meshes that may naturally compress finite element data, in Figure~\ref{fig: compressible meshes}.
Structured meshes and meshes with repeated patterns and templates will clearly have cells where finite element data may be reused.
Extruded meshes with fixed extrusion widths benefit greatly, as the extrusion layers are translations of each other.
While not shown in the figure, regularly refined grids may reap similar benefits.
\begin{figure}
    \centering
    \resizebox{!}{3.8cm}{
    \begin{tikzpicture}
    \begin{axis}[
    unit vector ratio*=1 1 1,
    xmin=0, xmax=1,
    ymin=0, ymax=1,
    xtick={0,0.008,0.027,0.055,0.093,0.138,0.192,0.254,0.323,0.399,0.482,0.572,0.669,0.773,0.883,1},
    ytick={0,0.008,0.027,0.055,0.093,0.138,0.192,0.254,0.323,0.399,0.482,0.572,0.669,0.773,0.883,1},
    grid=both,
    major grid style={black},
    minor grid style={black},
    xticklabel=\empty,
    yticklabel=\empty
    ]
    \end{axis}
    \end{tikzpicture}
    } %
    \hfill
    \resizebox{!}{3.8cm}{\includegraphics[trim=5.83cm 3.95cm 7.3cm 4.2cm,clip]{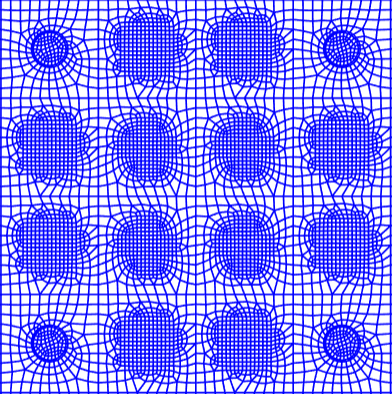}} %
    \hfill
    \resizebox{!}{3.8cm}{\includegraphics{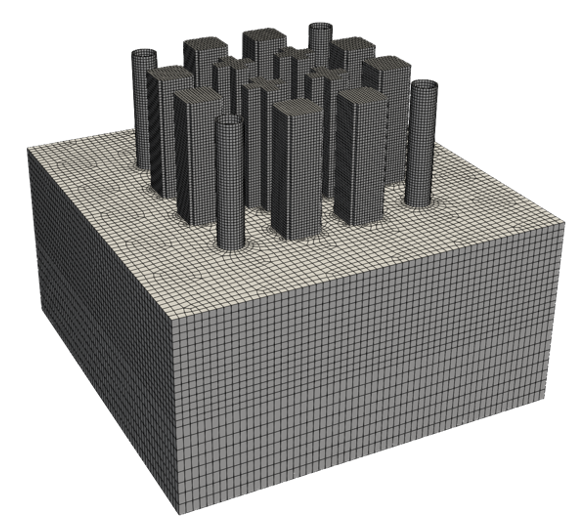}} %
    \caption{Three examples of mesh classes with redundancy: (left) structured/graded, (middle) geometrically patterned, and (right) extruded.  The latter two meshes comprise large clusters of identical cells.  The cells in the graded mesh have shapes that are identical up to a dimensional scaling.}
    \label{fig: compressible meshes}
\end{figure}

The \textit{compression ratio} of a mesh is the ratio of
the number of saved Jacobian evaluations and the total number of mesh cells.
In the case of a structured grid of equal shapes, the Jacobian needs be evaluated only once, as all cells are translations of each other.
In that case, the compression ratio is $\frac{n-1}{n}$, where $n$ is the total number of mesh cells.
Note the compression ratio can never be 100\%, as assembling a finite element operator always requires basis evaluations on at least one cell.
Some concrete examples of compression ratios are shown in Figure~\ref{fig: mesh compression}.%
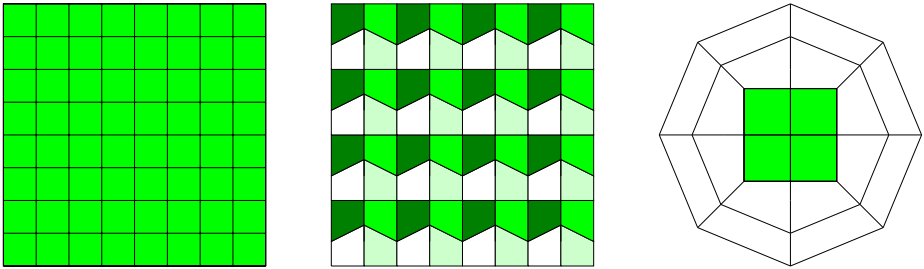
\begin{figure}
  \centering
  \resizebox{!}{3.5cm}{
  \begin{tikzpicture}
  \draw[step=0.5,fill=green] (0,0) grid (4,4) rectangle (0,0);
  \begin{scope}[xshift=5cm]
  \foreach \y in {0,1,2,3}{
  \foreach \x in {0,1,2,3}{
    \draw[fill=white] (\x,\y) -- (\x+0.5,\y) -- (\x+0.5,\y+0.625) -- (\x,\y+0.375) -- (\x,\y);
    \draw[fill=green!20!white] (\x+0.5,\y) -- (\x+1,\y) -- (\x+1,\y+0.375) -- (\x+0.5,\y+0.625) -- (\x+0.5,\y);
    \draw[fill=green!50!black] (\x,\y+0.375) -- (\x+0.5,\y+0.625) -- (\x+0.5,\y+1) -- (\x,\y+1) -- (\x,\y+0.375);
    \draw[fill=green] (\x+0.5,\y+0.625) -- (\x+1,\y+0.375) -- (\x+1,\y+1) -- (\x+0.5,\y+1) -- (\x+0.5,\y+0.375);
  }}
  \end{scope}
  \begin{scope}[xshift=12cm,yshift=2cm]
  \foreach \r in {1.5,2}{
  \foreach \t in {0,45,90,135,180,225,270,315}{%
    \draw ({\r*sin(\t)},{\r*cos(\t)}) -- ({\r*sin(\t+45)},{\r*cos(\t+45)});
    \draw ({\r*sin(\t)},{\r*cos(\t)}) -- ({(\r-0.5)*sin(\t)},{(\r-0.5)*cos(\t)});
  }
  }
  \draw[step=0.707,fill=green] (-0.707,-0.707) grid (0.707,0.707) rectangle (-0.707,-0.707);
  \draw (-1,0) -- (1,0); %
  \draw (0,-1) -- (0,1); %
  \end{scope}
  \end{tikzpicture}}
\caption{Illustration of (left) a structured square mesh with one unique shape and a lossless compression ratio ($\varepsilon = \varepsilon_0$) of $98.4\%$ (middle) a trapezoidal mesh with only four unique shapes and a lossless compression ratio of $93.75\%$, and (right) a circle mesh with a lossless compression ratio of $15\%$. Cells with a similar shape within a single mesh are highlighted with the same color.}
\label{fig: mesh compression}
\end{figure}

Due to floating point error, Jacobian comparisons based on equality are not practical.
Instead, our algorithm uses a tolerance $\varepsilon>0$.
A list of ``unique'' cells is constructed by starting with an empty list
and looping over each mesh cell $T_i$, which is compared to each cell $T_j$
in the list of unique cells.
A distance $d(T_i,T_j)$ is computed using the cell Jacobians,
\begin{align*}
  d(T_i,T_j)
  &=
  \frac{\sqrt{\sum_{k=1}^q w_k \| J_{T_i}(\hatbx_k) - J_{T_j}(\hatbx_k) \|_{F}^2}}{\sqrt{\sum_{k=1}^q w_k \| J_{T_j}(\hatbx_k) \|_{F}^2}} ,
\end{align*}
where $\hatbx_k$ denote the quadrature points on the reference domain, with weights $w_k$, $k=1,2,\dots,q$, and $\|\cdot\|_F$ is the Frobenius norm.
If $d(T_i,T_j) < \varepsilon$, then a match is found, and basis evaluations on $T_i$ are replaced by the basis evaluations on $T_j$.
Later, we will discuss different types of compression, based on $\varepsilon$.
Let $\varepsilon_0$ be the largest acceptable compression tolerance
for a given application without a significant impact on accuracy.
We say that the compression is \textit{lossless} when $\varepsilon\leq\varepsilon_0$, and if $\varepsilon>\varepsilon_0$
we refer to the compression as \textit{lossy}.
Algorithm~\ref{algo: mesh compression} summarizes our proposed scheme.
\begin{algorithm}[H]
\caption{Construction of the mesh dictionary} \label{algo: mesh compression}
\begin{algorithmic}[1]
\State \textbf{Input:} Mesh $\mesh=\{T_1, T_2, \dots, T_N \}$, tolerance $\varepsilon$.
\State $\mathcal{B}_\varepsilon \leftarrow \{\}$, $\psi_\varepsilon \leftarrow\vec{0}\in\N^N$
\For{$i=1,2,\dots,N$}
  \State match$\leftarrow$\texttt{false}
  \For{$T_j\in\mathcal{B}_\varepsilon$}
    \If{$d(T_i,T_j) <\varepsilon$}
      \State match$\leftarrow$\texttt{true}; $\psi_\varepsilon(i) \leftarrow j$; \textbf{break}
    \EndIf
  \EndFor
  \If{match=\texttt{false}}
    \State append $T_i$ to $\mathcal{B}_\varepsilon$; $\psi_\varepsilon(i) \leftarrow |\mathcal{B}_\varepsilon|$
  \EndIf
\EndFor
\State \textbf{Output:} Dictionary of cells, $\mathcal{B}_\varepsilon$, with lookup indices $\psi_\varepsilon\in\N^N$.
\end{algorithmic}
\end{algorithm}
This algorithm is a greedy approach as it loops over cells sequentially
while searching for matches.
While it is not guaranteed to find the best match for a cell,
it finds the first match for a cell within the specified tolerance.
Many performance-related optimizations are possible, however they are beyond the scope of this paper.
Because the output of Algorithm~\ref{algo: mesh compression} depends on the tolerance $\varepsilon$, the compression ratio, given by $\frac{N-|\mathcal{B}_\varepsilon|}{N}$, is a function of $\varepsilon$.
We will refer to the curve relating the compression ratio to $\varepsilon$ as the \textit{compression curve}.
Figure \ref{fig: compression curve} demonstrates the curve for a circle mesh.
The algorithm is implemented in the FEM software library \mrhyde~\cite{mrhyde2023github}.
\begin{figure}
\includegraphics[height=0.4\textwidth]{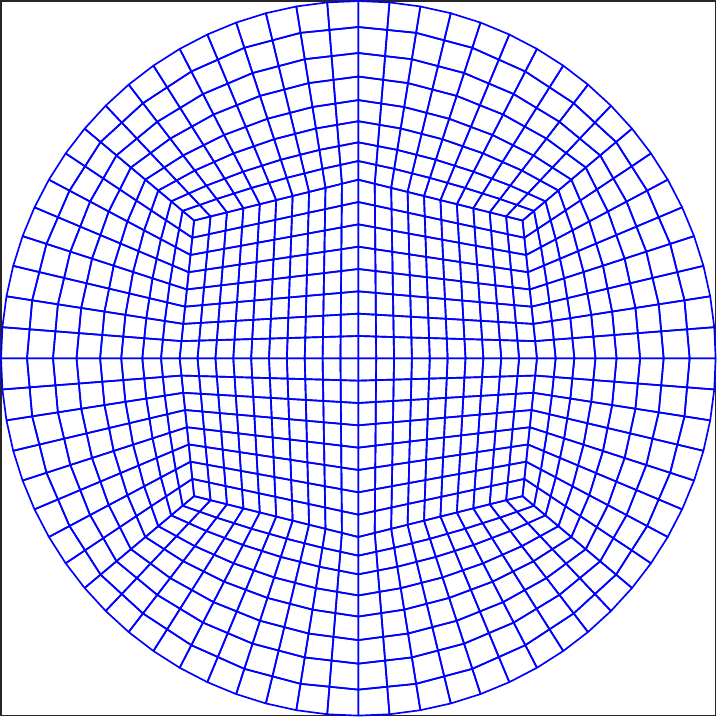}
\hfill
\includegraphics[height=0.4\textwidth]{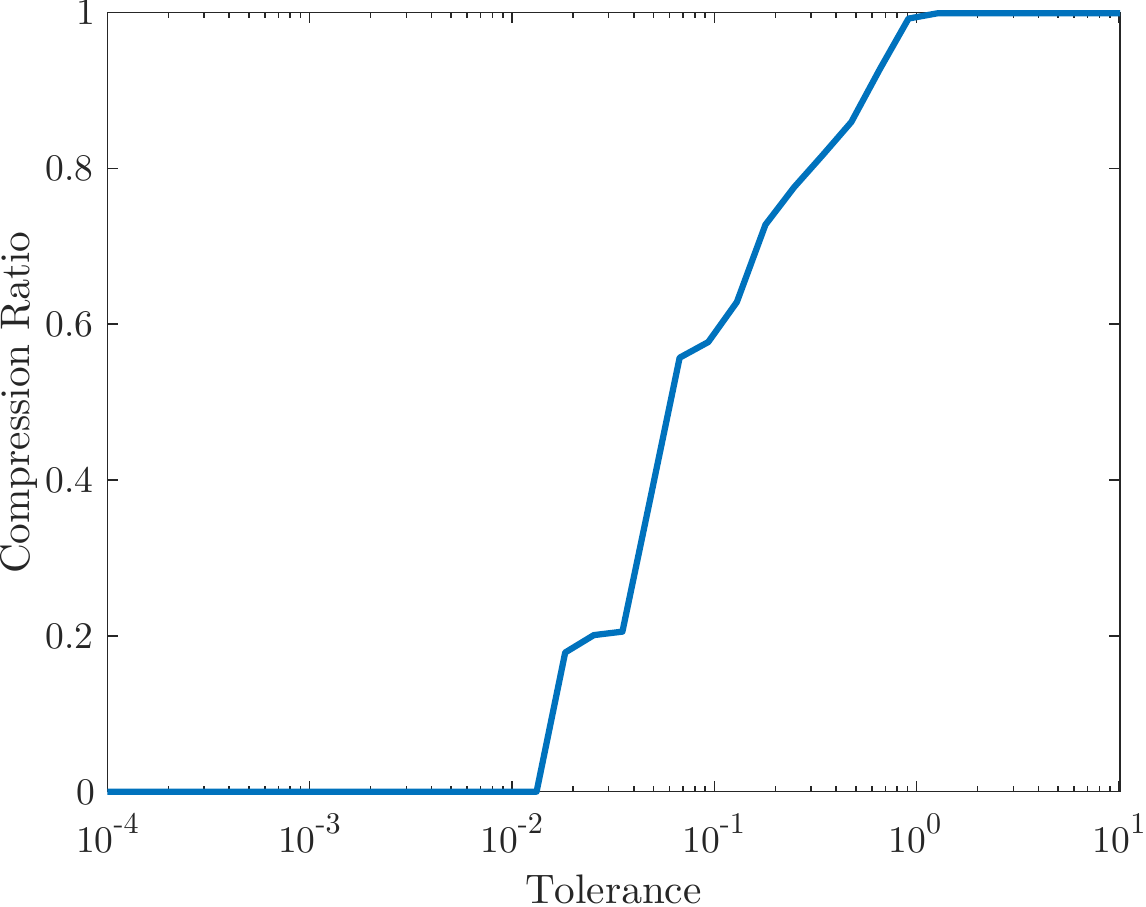}
\caption{Compression curve for a circle mesh; no compression is observed for $\varepsilon < 10^{-2}$.}
\label{fig: compression curve}
\end{figure}

After the compression algorithm is applied, the physical basis values
for each basis function at each quadrature point are only calculated for
the shapes in $\mathcal{B}_\varepsilon$.
This means the compression ratio for a mesh is directly related to the
amount of compression obtained on the physical finite element basis values.
For near-zero values of $\varepsilon$, this compression algorithm
introduces near-zero errors in the discretization, as the reused
basis values are nearly equal.
While it is possible to use larger values of $\varepsilon$, this introduces
error in the discretization, which we will analyze in the future.
Memory profiles and timing data for a heat transfer simulation in \mrhyde{} are shown in Figure~\ref{fig: compression speedup}.
We use a uniform $32,000\times32,000$ mesh of square cells and a compression tolerance of~$\varepsilon=10^{-10}$.
The simulations are run on a single compute node, in serial, for four time steps, where each brief memory decrease in the final flat region of the profile signals the start of a new time step.
Since the mesh is uniform, the basis compression reduces the memory cost of a billion sets of basis function data to a single set of basis function data, without negatively impacting the runtime.
 \begin{figure}
 \centering
 \includegraphics[height=0.4\textwidth]{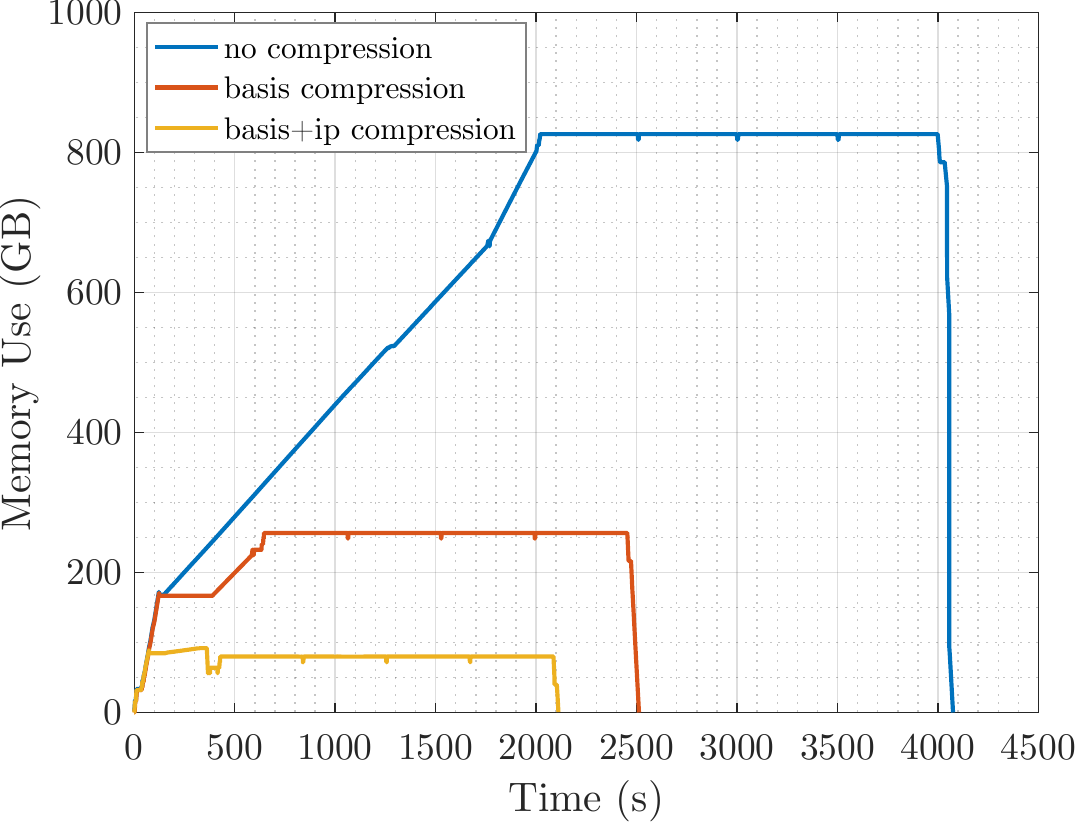}
 \caption{Memory use and runtime for a 1.024 billion-cell heat transfer simulation.}
 \label{fig: compression speedup}
 \end{figure}
An additional curve in Figure~\ref{fig: compression speedup} shows
integration point compression in addition to basis compression,
marked by ``basis+ip compression.''
In order to compress the evaluation of the residual of the heat equation, information at quadrature points must be retained.
This is accomplished by performing compression on the $x$ and $y$ coordinates individually and storing every quadrature point as a pair of integers, to index the unique entries for both.
Combining the integration point compression with the finite element basis compression reduces the maximum memory use from over 820GB to 92GB, while maintaining computational performance.

Many finite element meshes lack redundancy and do not lend themselves immediately to finite element data compression.
Additionally, meshing constraints and intricate geometry often restrict the classes of meshes that are viable for certain applications.
To extend the benefits of basis compression to such scenarios, we now focus on r-adaptive mesh modifications, with the goal to enhance redundant structure.

\section{Enhancing mesh redundancy: Problem formulation}
\label{sec:ProblemFormulation}

Keeping the performance results from the previous section in mind, we now propose to modify a given mesh without changing its topology, with the goal to increase its compression ratio for a given $\varepsilon$.
We focus on the quadrilateral case with straight line geometry, and note that our optimization formulation can be extended to three dimensions and other cell shapes and geometries.
Our choice of a quadrilateral mesh, rather than a triangular mesh, is to avoid the loss of generality inherent in specializations to simplicial meshes.
Let $\mesh = \{ T_1,T_2,\dots,T_{N} \}$ be a given mesh of a domain $\Omega\subset\R^d$, $d=2$ with $n_p$ vertices $\hat{\bp} = \{ \hat{p}_1, \hat{p}_2, \dots, \hat{p}_{n_p} \}$, and let $\mathcal{B}_\varepsilon$ be the dictionary obtained by applying Algorithm~\ref{algo: mesh compression} to $\mesh$ for a tolerance $\varepsilon$.
Our goal is to produce a new mesh $\mathcal{T}_{h,r}$ such that the resulting $\mathcal{B}_{\varepsilon,r}$ is smaller.
Let $\bp = \{ p_1,p_2,\dots,p_{n_p} \}$ be the list of vertex variables to optimize, where $p_i = (x_i,y_i)$, $i=1,2,\dots,n_p$.
Let the double-indexing $p_{i,j} = (x_{i,j},y_{i,j})$ denote vertex $j$ of mesh cell $i$, $i=1,2,\dots,N$, $j=1,2,\dots,n_{cv}$, with $n_{cv}$ denoting the number of vertices in each cell. We assume a fixed mesh topology given by $\mesh$ and the counterclockwise cell orientation.

Let the reference triangle be given by points $(0,0)$, $(1,0)$, $(0,1)$.
Let the reference square be given by points $(0,0)$, $(1,0)$, $(1,1)$, $(0,1)$.
Figure \ref{fig: reference triangle square}
shows the reference shapes for triangles and squares.
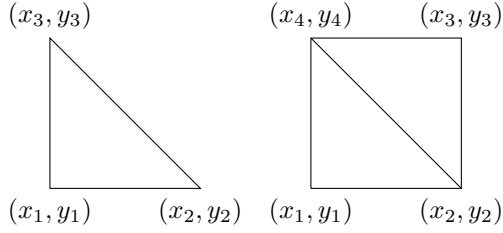
\begin{figure}
  \centering
  \begin{tikzpicture}[domain=0:3]
    \draw (0,0) -- (2,0) -- (0,2) -- (0,0);
    \node at (0,-0.3) {$(x_1,y_1)$};
    \node at (2,-0.3) {$(x_2,y_2)$};
    \node at (0,2.3) {$(x_3,y_3)$};
  \end{tikzpicture}
  \begin{tikzpicture}[domain=0:3]
    \draw (0,0) -- (2,0) -- (2,2) -- (0,2) -- (0,0);
    \draw (2,0) -- (0,2);
    \node at (0,-0.3) {$(x_1,y_1)$};
    \node at (2,-0.3) {$(x_2,y_2)$};
    \node at (2,2.3) {$(x_3,y_3)$};
    \node at (0,2.3) {$(x_4,y_4)$};
  \end{tikzpicture}
  \caption{Reference square and triangle geometries, both subsets of $[0,1]^2$.}
  \label{fig: reference triangle square}
\end{figure}
For simplicial $T$ with straight line geometry, $J_T$ is a constant.
For example, $J_T$ for the triangle in Figure~\ref{fig: reference triangle square} is given by
\begin{align*}
  J_{T}(\hat{\mathbf{x}})
  &=
  \left[
  \begin{array}{cc}
  (x_2-x_1) & (x_3-x_1)
  \\
  (y_2-y_1) & (y_3-y_1)
  \end{array}
  \right],
\end{align*}
regardless of $\hat{\mathbf{x}}$.
However, $J_T$ is not constant for quadrilaterals, even with straight line edges,
instead yielding
\begin{align*}
  J_{T}(\hat{\mathbf{x}})
  &=
  \left[
  \begin{array}{cc}
  (x_2-x_1) + (x_1-x_2+x_3-x_4)\hat{y} & (x_4-x_1) + (x_1-x_2+x_3-x_4)\hat{x}
  \\
  (y_2-y_1) + (y_1-y_2+y_3-y_4)\hat{y} & (y_4-y_1) + (y_1-y_2+y_3-y_4)\hat{x}
  \end{array}
  \right].
\end{align*}
This requires evaluating $J_T$ at many
quadrature points due to the spatial dependence.
In order to circumvent this issue, we split the quadrilateral into
two triangles using the vertex index triples $(1,2,4)$ and $(3,4,2)$,
as depicted in Figure \ref{fig: reference triangle square},
and compute Jacobians for both triangles.
On a general mesh cell $T_i$, $1\leq i \leq N$, we define
\begin{align*} 
  c_{i,x2}
  &=
  x_{i,2} - x_{i,1},
  & %
  d_{i,x4}
  &=
  x_{i,4} - x_{i,3},
  \\ %
  c_{i,x4}
  &=
  x_{i,4} - x_{i,1},
  & %
  d_{i,x2}
  &=
  x_{i,2} - x_{i,3},
  \\ %
  c_{i,y2}
  &=
  y_{i,2} - y_{i,1},
  & %
  d_{i,y4}
  &=
  y_{i,4} - y_{i,3},
  \\ %
  c_{i,y4}
  &=
  y_{i,4} - y_{i,1},
  & %
  d_{i,y2}
  &=
  y_{i,2} - y_{i,3}.
\end{align*}
The two triangle Jacobians composed of these scalars, obtained by splitting the cell~$T_i$, define a computationally efficient proxy for the Jacobian of the quadrilateral.
We then define the row vector $K_i(\bp)$, $i=1,2,\dots,N$, consisting of the Jacobian entries,
\begin{align}
  K_i(\bp)
  &=
  [ c_{i,x2}, c_{i,x4}, d_{i,x4}, d_{i,x2}, c_{i,y2}, c_{i,y4}, d_{i,y4}, d_{i,y2} ].
  \label{eq: jacobian vector}
\end{align}
For each cell $T_i$ in the mesh, we also associate a ``target'' shape $\mu_i$, $i=1,2,\dots,N$, which is a row vector of eight entries as well, analogous to $K_i(\bp)$ and defined in Section~\ref{sec:targets}.
Each of these row vectors are rows of larger matrices $K(\bp),\mu\in\R^{N\times8}$.
Our goal is to reduce the misfit between $K_i(\bp)$ and $\mu_i$, for all $i=1,2,\dots,N$, i.e., the misfit between $K(\bp)$ and~$\mu$, while also avoiding significant mesh distortions that cause low cell-shape quality and mesh tangling. 

\subsection{Formulation of the objective function}
There are several ways in which to measure misfit between $K(\bp)$ and~$\mu$.
For the remainder of the paper, we focus on quantities of interest that promote lossless finite element basis compression.
Ideally, we would like to solve a sparse matching problem between the rows of $K(\bp)$ and the rows of $\mu$, which would maximize the number of zero rows in $K(\bp) - \mu$.
Approaches enforcing the so-called \emph{group sparsity}, such as those using the mixed $\ell_1/\ell_q$ norm, $q>1$, have been suggested~\cite{bach2012structured}, however they lead to non-smooth quantities of interest of the form
\begin{align} \label{eq:mixed}
  \sum_{i=1}^N \| K_i(\bp) - \mu_i \|_q.
\end{align}
The non-smoothness, induced by the norm, is difficult to handle algorithmically---we are not aware of efficient optimization algorithms for the minimization of~\eqref{eq:mixed} that would scale with the number of mesh cells,~$N$.
In contrast, the functional
\begin{align} \label{eq:ell2squared}
  \sum_{i=1}^{N} \| K_i(\bp) - \mu_i \|_{2}^2,
\end{align}
which utilizes the square of the 2-norm, is smooth and relatively easy to minimize using derivative-based optimization methods.
Unfortunately, any notion of group sparsity may be lost, as the overall misfit
is spread across all mesh cells.
Minimizing~\eqref{eq:ell2squared} typically produces no zero-misfit terms whatsoever, i.e., no lossless compression, unless $K_i(\bp) - \mu_i = 0$ is achievable at the optimum for all $i=1,2,\dots,N$, which is possible only for very special domains and meshes.
However, this challenge also motivates a key idea:
\begin{gather*}
    \text{identify a subset of cell indices,}
    \;\; \mathcal{Z} \subseteq \{1, 2, \dots, N\},\\
    \text{such that} \;\; K_i(\bp) - \mu_i = 0 \;\;
    \text{is achievable for all} \;\; i \in \mathcal{Z}.
\end{gather*}
Therefore, given a set $\mathcal{Z}$, the functional to minimize is
\[
   \sum_{i \in \mathcal{Z}} \| K_i(\bp) - \mu_i \|_{2}^2,
\]
leading to our primary ``lossless'' objective functional
\begin{equation} \label{eq:ell2lossless}
  L(\bp) = \sum_{i = 1}^{N} \alpha_i \| K_i(\bp) - \mu_i \|_{2}^2,
  \quad \alpha_i \in \{0,1\}, \quad \text{for} \; i=1, 2, \dots, N,
\end{equation}
where $\alpha_i = 1$ if $i \in \mathcal{Z}$ and $\alpha_i = 0$ otherwise.
We note its obvious connection to the concept of group sparsity.
Finally, to identify the set $\mathcal{Z}$ we will also consider a relaxed version of~\eqref{eq:ell2lossless}, namely
\begin{equation} \label{eq:ell2relaxed}
  L_r(\bp) = \sum_{i = 1}^{N} \alpha_i \| K_i(\bp) - \mu_i \|_{2}^2,
  \quad \alpha_i \in \mathbb{R}^+_0, \quad \text{for} \; i=1, 2, \dots, N.
\end{equation}

\begin{remark}
There is one challenge when utilizing this matching approach with a target $\mu$.
Permutations of local node orderings in mesh cells cause modifications to $K(\bp)$.
For instance, consider two identical unit squares $[0,1]^2$ with the local vertex orderings specified in Figure \ref{fig: square node permutation}.
We have $K_1(\bp) = [1,0,-1,0,0,1,0,-1]$ and
$K_2(\bp) = [0,-1,0,1,1,0,-1,0]$.
This results in $\|K_1(\bp)-K_2(\bp)\|_{\ell_2}^2 = 64$, rather than zero.
For this reason, node permutations can adversely affect the compressibility of the finite element basis functions.
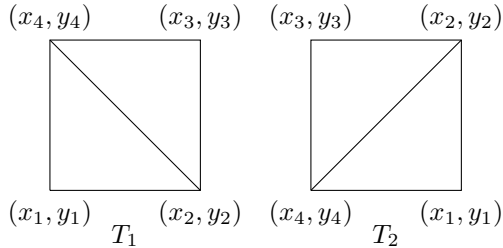
\begin{figure}
  \centering
  \begin{tikzpicture}[domain=0:3]
    \draw (0,0) -- (2,0) -- (2,2) -- (0,2) -- (0,0);
    \draw (2,0) -- (0,2);
    \node at (0,-0.3) {$(x_1,y_1)$};
    \node at (2,-0.3) {$(x_2,y_2)$};
    \node at (2,2.3) {$(x_3,y_3)$};
    \node at (0,2.3) {$(x_4,y_4)$};
    \node at (1,-0.6) {$T_1$};
  \end{tikzpicture}
  \begin{tikzpicture}[domain=0:3]
    \draw (0,0) -- (2,0) -- (2,2) -- (0,2) -- (0,0);
    \draw (0,0) -- (2,2);
    \node at (0,-0.3) {$(x_4,y_4)$};
    \node at (2,-0.3) {$(x_1,y_1)$};
    \node at (2,2.3) {$(x_2,y_2)$};
    \node at (0,2.3) {$(x_3,y_3)$};
    \node at (1,-0.6) {$T_2$};
  \end{tikzpicture}
  \caption{Squares with differing local node orderings.}
  \label{fig: square node permutation}
\end{figure}
Permutation issues may appear even in regular refinements of consistently oriented meshes, as some refinement algorithms permute the local order in each of a cell's children.
To avoid issues related to this, we pre-process the mesh to place the first node in the ``lower-left'' corner, with the following local nodes proceeding counter-clockwise.
By doing so, for quadrilaterals we rule out three permutations.
This is demonstrated in Figure \ref{fig: circlemesh permutations}, where a circular mesh is re-oriented to be more amenable to compression and optimization.
There are works such as \cite{agelek2017orienting}, showing that meshes may be oriented consistently in two dimensions as a consequence of the Jordan curve theorem, which we have considered as an alternative.
Still, a deeper study of local mesh orientation strategies remains beyond the scope of this paper.
\end{remark}

\begin{figure}
  \hfill
  \includegraphics[height=0.35\textwidth]{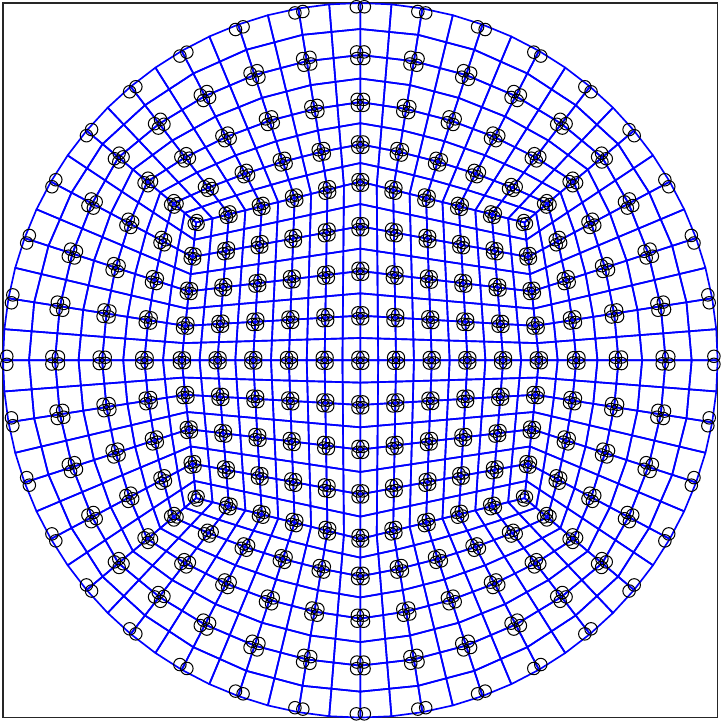}
  \hfill
  \includegraphics[height=0.35\textwidth]{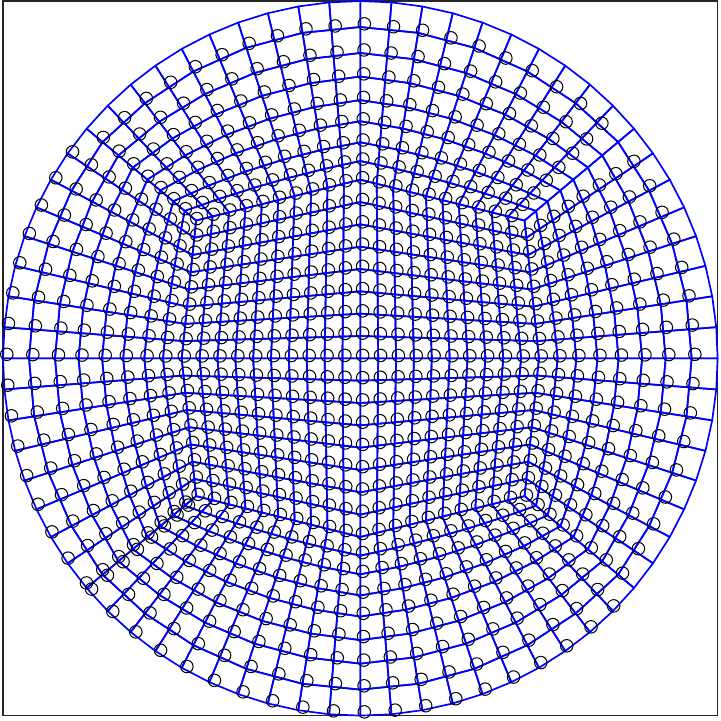}
  \hspace{7ex}
  \caption{A circle mesh with the first local vertex marked by a small black circle in each cell, (left) before and (right) after the ``lower-left'' modification.}
  \label{fig: circlemesh permutations}
\end{figure}

\subsection{Formulation of the constraints}
We require that the mesh modifications, due to minimizing our primary objective function~\eqref{eq:ell2lossless}, maintain the quality of the finite element approximation given by the original mesh.
Among other requirements, mesh cells should not become tangled or inverted.
Therefore, we consider restricting the volume change for each mesh cell.
Similar constraints were considered in~\cite{ridzal2016meshcorrection},
and were found to successfully mitigate cell distortion. 
In this context, we extend the volume \emph{equality} constraints from~\cite{ridzal2016meshcorrection} to \emph{inequality} constraints.
This extension requires an extension of the composite-step sequential quadratic programming (SQP) algorithm~\cite{heinkenschloss2014matrix}, used in~\cite{ridzal2016meshcorrection}, to the recently developed augmented Lagrangian equality-constrained SQP (ALESQP) algorithm~\cite{antil2023alesqp}.
Additionally, boundary nodes and nodes defining important geometric features should remain fixed.
The latter constraint is handled easily by eliminating variables (vertices) from the optimization problem, and we do not elaborate on it further.
We note that additional constraints, such as cell convexity constraints, may be useful in certain applications, and that they are easily included in our formulation, e.g., following~\cite[Eq.~(10)]{ridzal2016meshcorrection}.
The ALESQP algorithm, discussed in Section~\ref{sec:Algorithms}, is well equipped to handle such additional inequality constraints, and it enables a more efficient and more robust scheme compared to solving the ad-hoc logarithmic barrier penalty formulation~\cite[Eq.~(21)]{ridzal2016meshcorrection} using SQP.
To keep our presentation focused, we do not include convexity constraints in the formulation.

We formulate the volume constraints as follows.
Let $v_i(\bp)$ be the volume of cell~$i$, $i=1, 2, \dots N$,
given by the well-known shoelace formula
\begin{align}
  v_i(\bp)
  &=
  \frac{1}{2} \sum_{j=1}^{n_{cv}} x_{i,j} y_{i,j+1} - x_{i,j+1} y_{i,j},
  \quad
  i=1,2,\dots,N,
  \label{eq: shoelace}
\end{align}
where $n_{cv}=4$ for quadrilaterals, and index $n_{cv}+1=5$ is associated with $5 \pmod{n_{cv}}$, i.e., index $1$.
Let $\vmin = \min_{i \in \{1, 2, \dots N\}} v_i(\hat{\bp})$
and $\vmax = \max_{i \in \{1, 2, \dots N\}} v_i(\hat{\bp})$ be the minimum and maximum cell volumes on the initial mesh $\mesh$, and let $0<\gamma<1$ be a fixed factor that controls how much the mesh cells may grow or shrink.
Depending on the application, for each cell we define the lower and upper volume bounds, $v^{\text{lo}}_i$ and $v^{\text{up}}_i$, $i=1,2,\dots, N$, respectively, in two ways: using global quantities, where
\begin{equation} \label{eq:globalbounds}
  v^{\text{lo}}_i = (1-\gamma)\vmin \quad \text{and} \quad
  v^{\text{up}}_i = (1+\gamma)\vmax , \quad i = 1, 2, \dots N, 
\end{equation}
or using local quantities, where
\begin{equation} \label{eq:localbounds}
  v^{\text{lo}}_i = (1-\gamma) v_i(\hat{\bp}) \quad \text{and} \quad
  v^{\text{up}}_i = (1+\gamma) v_i(\hat{\bp}), \quad i = 1, 2, \dots N.
\end{equation}
We collect the volume functionals and quantities into vectors,
$v(\bp) : \mathbb{R}^{n_p} \rightarrow \mathbb{R}^N$,
$v^{\text{lo}} \in \mathbb{R}^N$, and
$v^{\text{up}} \in \mathbb{R}^N$.
This completes the formulation of the optimization problem.
In summary, we seek $\bp \in \mathbb{R}^{n_p}$ that solves
\begin{subequations}
\label{eq:optproblem}
\begin{align}
  \underset{\bp}{\text{minimize}} &\;\;  L(\bp) \label{eq:optproblemobj} \\
  \text{subject to} &\;\;
    v^{\text{lo}} \le v(\bp) \le v^{\text{up}} .  \label{eq:optproblemcon}
\end{align}
\end{subequations}
Next we show that problem~\eqref{eq:optproblem} has a solution.

\begin{lemma}
The feasible set defined in~\eqref{eq:optproblemcon} is nonempty for both definitions of the bounds, \eqref{eq:globalbounds} and~\eqref{eq:localbounds}.
\end{lemma}
\begin{proof}
The original mesh configuration, given by the points $\bp=\hatbp$, is inside the feasible regions.
To see this, recall $0 < \gamma < 1$ and note that, for $i=1, 2, \dots, N$,
\begin{align*}
  (1-\gamma) \vmin \le (1-\gamma) v_i(\hat{\bp})
  \le v_i(\hatbp)
  \le (1+\gamma) v_i(\hat{\bp}) \le (1+\gamma) \vmax.
\end{align*}
\end{proof}

\section{Enhancing mesh redundancy: Algorithms}
\label{sec:Algorithms}
There are three key computational components in our proposed lossless mesh optimization: the solution of the optimization problem~\eqref{eq:optproblem}, the computation of the cell targets $\mu$ and the identification of the index set $\mathcal{Z}$.
We solve~\eqref{eq:optproblem} using the ALESQP algorithm~\cite{antil2023alesqp}.
We compute cell targets using k-medoid clustering.
We compute the index set using a bracketing procedure.
The bracketing, clustering and the solution of~\eqref{eq:optproblem} are
integrated in an algorithm whose performance scales linearly with
the number of mesh cells, $N$.

\subsection{Mesh optimization using ALESQP}
\label{sec:SolutionMethod}

The ALESQP algorithm uses an augmented Lagrangian penalty formulation to handle inequality constraints, while relying on the repeated solution of equality-constrained optimization subproblems with the composite-step SQP approach developed in~\cite{heinkenschloss2014matrix}.
General inequality constraints are handled in ALESQP using slack
variables.
Specifically, we define slack variables $s \in \mathbb{R}^N$, $s_i = v_i(\bp)$, $i=1,2,\dots,N$,
and replace~\eqref{eq:optproblem} with
\begin{subequations}
\label{eq:optslack}
\begin{align}
  \underset{\bp,s}{\text{minimize}}  &\;\;  L(\bp) \\
  \text{subject to} &\;\;  v(\bp) - s = 0,
  \\
  &\;\;
  v^{\text{lo}} \le s \le v^{\text{up}}.
\end{align}
\end{subequations}
The subproblems solved by ALESQP penalize inequality constraints, such as the bounds on slack variables, and maintain equality constraints explicitly.  They read
\begin{subequations} \label{eq:subAL}
\begin{align}
  \underset{\bp,s}{\text{minimize}}  &\;\;  L(\bp) + 
  \frac{1}{2r} \left\| r\left( \frac{\lambda}{r} + s -
  \max \Bigl( \min \bigl( \frac{\lambda}{r} + s, v^{\text{lo}} \bigr), v^{\text{up}} \Bigr) \right) \right\|_2^2 \label{eq:subALobj}
  \\
  \text{subject to} &\;\;  v(\bp) - s = 0,  \label{eq:subALcon}
\end{align}
\end{subequations}
where $r>0$ is a scalar computed by ALESQP, $\lambda \in \mathbb{R}^N$ is a Lagrange multiplier corresponding to the slack variables $s$,
and $\min$ and $\max$ are applied elementwise.
We note that the objective function~\eqref{eq:subALobj} is differentiable
despite the presence of non-smooth $\min$ and $\max$ functions.
Subproblem~\eqref{eq:subAL} is solved using the composite-step SQP method, in which the trial step is split into a quasi-normal step, which aims to satisfy the linearized equality constraint, and a tangential step, which aims to improve optimality, while maintaining the linearized feasibility gained by the quasi-normal step.
The tangential step computation involves a quadratic programming (QP) subproblem and dominates the total computational cost.
The QP subproblem is solved using a projected conjugate gradient (CG) method, where the CG operator is related to the Hessian of the Lagrangian for~\eqref{eq:subAL}, and the projections aim to place the iterates in the null space of the linearized constraints~\eqref{eq:subALcon}.
The key linear algebra kernel in all components of the composite-step SQP method, and therefore ALESQP, is the solution of an augmented system involving only the Jacobian of the equality constraints, and its transpose.
In each SQP stage, the augmented system is an indefinite symmetric two-by-two block system where the $(1,1)$ block is the identity, the $(1,2)$ block is the Jacobian of the equality constraint~\eqref{eq:subALcon}, and the $(2,2)$ block is zero.
Since we consider slack variables, we break the blocks into the $\bp$ variables and the $s$ variables, yielding the augmented system
\begin{align}
  \renewcommand{\arraystretch}{1.33}
  \begin{bmatrix}
    I & 0 & A_{\bp}^\mathsf{T} \\
  0 & I & A_{s}^\mathsf{T} \\
  A_{\bp} & A_{s} & 0
  \end{bmatrix}
  \begin{bmatrix}
  \bp \\
  s \\
  \lambda
  \end{bmatrix}
  &=
  \renewcommand{\arraystretch}{1.33}
  \begin{bmatrix}
  b_{\bp} \\
  b_s \\
  b_{\lambda} \\
  \end{bmatrix} ,
  \label{eq: linear system}
\end{align}
where $A_{\bp}$ corresponds to the Jacobian $v'_{\bp}(\bp)$ and $A_s$ corresponds to $-I$, the negative identity.
As $v(\bp)$ is defined by the shoelace formula, the entries of $A_{\bp}$ take the form $y_{i,k+1} - y_{i,k-1}$ and $x_{i,k-1} - x_{i,k+1}$.
Interestingly, the structure of $A_{\bp}$ appears similar to a discrete divergence on the mesh.
Similarly, $A_{\bp}^\mathsf{T}$ resembles a discrete gradient.
Taking this further, the Schur complement of the system matrix in~\eqref{eq: linear system} is
\[
  S = A_{\bp}A_{\bp}^\mathsf{T} + I ,
\]
where the $A_{\bp}A_{\bp}^\mathsf{T}$ matrix resembles a discrete Laplacian.
Therefore, the augmented system matrix is spectrally related to a perturbation of the identity by a discrete Laplacian, which is extremely well conditioned.
In other words, the augmented system is easily solved using Krylov methods like MINRES~\cite{paige1975solution} or GMRES~\cite{saad1986gmres}, without the need for preconditioning.
Our numerical results confirm this.
For problems with large domain lengthscales, where the term $A_{\bp}A_{\bp}^\mathsf{T}$ may dominate, multigrid methods are very effective in approximating the application of $S^{-1}$.

\subsection{Computing cell targets} \label{sec:targets}

To compute the cell targets $\mu$, we studied $k$-means and $k$-medoids~\cite{park2009simple} algorithms.
The latter proved to be the appropriate choice for our primary goal of lossless compression.
We recall that medoids are \emph{representative} objects of a data cluster,
which minimize the distance to all other data points in the cluster
using some choice of a distance metric.
To tailor the performance of $k$-medoids to our setting, we developed the approach shown in Algorithm~\ref{algo: clustering}.
We initialized our $k$-medoids clustering algorithm by drawing $k$ samples randomly from $N$ data points, without replacement,
the indices of which are denoted by the permutation $\sigma^N_k$.
After performing clustering on $K(\bp)$ and obtaining the cluster assignments $\psi$ and indices
of cluster medoids $C$, the shape targets are defined by the corresponding medoid for each data point,
$\mu_i = K_{C_{\psi_i}}(\bp)$, $i=1,2,\dots,N$.

\begin{algorithm}[htb!]
\caption{Computation of shape clusters} \label{algo: clustering}
\begin{algorithmic}[1]
\State \textbf{Input:}   Jacobians $J\in\R^{N\times8}$, number of clusters $k$.
\State \textbf{Initialization:} $\psi \leftarrow \vec{0}\in\N^N$, $\mbox{conv} \leftarrow \texttt{false}$, draw permutation $C \leftarrow \sigma^N_k \in \N^k$.
\While{not $\mbox{conv}$}
  \For{$i=1,2,\dots,N$} // Assign $\psi_i$ to closest cluster index based on $\| \cdot \|_{\ell_2}^2$.
    \State $d \leftarrow \vec{0}\in\R^k$
    \For{$j=1,2,\dots,k$}
      \State $d_j \leftarrow \|J(i,:) - J(C_j,:)\|_{\ell_2}^2$
    \EndFor
    \State $\psi_i \leftarrow s : d_s = \min(d)$
  \EndFor
  \State $C_{\rm{old}} \leftarrow C$
  \For{$j=1,2,\dots,k$} // Assign $C_j$ to global index for cluster medoid.
    \State $\mathcal{I} = \{ l : \psi_l = j \}; \ m = |\mathcal{I}|$
    \State $v = \vec{0} \in \N^m; \ l=1$
    \For{$i=1,2,\dots,N$} // Construct cluster-to-data lookup $\vec{v}$.
      \If{$\psi_i = j$}
        \State $v_l \leftarrow i; \ l \leftarrow l+1$
      \EndIf
    \EndFor
    \State $J_{\rm{med}} \leftarrow \mbox{median}(J(v,:))$ // Column-wise median over $m$ cluster members.
    \State $d \leftarrow \vec{0}\in\R^m$
    \For{$l=1,2,\dots,m$}
      \State $d_l \leftarrow \|J(v_l,:) - J_{\rm{med}}\|^2_{\ell_2}$
    \EndFor
    \State $C_j \leftarrow v_s : d_s = \min(d)$ // Cluster medoid is closest member to $J_{\rm{med}}$.
  \EndFor
  \If{${C}_{\rm{old}} = {C}$} // Converged if assignments are unchanged.
    \State $\mbox{conv} \leftarrow \texttt{true}$
  \EndIf
\EndWhile
\State \textbf{Output:} Indices of cluster medoids, ${C}\in\N^k$, and cluster assignments, ${\psi}\in\N^N$.
\end{algorithmic}
\end{algorithm}

\subsection{Computing cell weights}
Our proposed algorithm consists of two phases: a fast ``ranking'' phase to determine the best \emph{candidate} cells for lossless compression, and a slower ``bracketing'' phase to precisely determine the cells that are losslessly compressible.
In the ranking phase, where we use the relaxed objective function $L_r$, we compute cell-wise weights $\alpha \in \mathbb{R}^N$ based on the normalized squares of cluster sizes.
For cell $i$, let its cluster be cluster $j$, and let $\beta_j$ be the number of members in cluster~$j$.
Then the weight assigned to all cells in cluster $j$ is given by
\begin{align} \label{eq:weights}
  \alpha_i = \frac{\beta_j^2}{\sum_j \beta_j^2}, \quad i=1,2,\dots N.
\end{align}
This weighting scheme prioritizes large clusters while sacrificing smaller clusters.
This allows the optimization algorithm, which utilizes a least-squares difference, to focus on the greatest gains possible.

In the bracketing phase, we employ a bisection-like procedure based on the percentage of the cells for which we conjecture near-zero misfit.
We begin with a guess of 50\%, and if the obtained objective function is small, after several rounds of clustering, we increase the percentage to 75\%; if it is large, we decrease it to 25\%.
The process is then repeated.
This approach allows us to zoom in, within 1\% accuracy, on the number of losslessly compressible cells in only a handful of iterations, due to its logarithmic complexity.
Given a guess $\eta \in \{1,2,\dots N\}$ of the number of losslessly compressible cells, we compute the weights $\alpha \in \{0, 1\}^N$ by sorting the cell misfits $\| K_i(\bp) - \mu_i \|_{2}^2$, $i=1,2,\dots N$, in ascending order, and computing a cell index permutation, $\mathcal{E}$, corresponding to the sorted order.
Then, for $i=1,2,\dots N$,
\begin{align} \label{eq:binary}
  \alpha_i = 
  \begin{cases}
  1 \quad \text{if} \; \mathcal{E}(i) \le \eta, \\
  0 \quad \text{otherwise}.
  \end{cases}
\end{align}
The complete r-adaptive mesh compression algorithm is given in Algorithm~\ref{algo: radapt}.

\begin{algorithm}[tbh!]
\caption{r-adaptive mesh compression} \label{algo: radapt}
\begin{algorithmic}[1]
\State \textbf{Input:}   Mesh $\mesh$ with points $\hat{\bp}$; compression tolerance $\varepsilon$; ALESQP tolerance $\text{tol}$.
\State \textbf{Initialization:} Set ALESQP parameters as in~\cite[p.~258]{antil2023alesqp} and set ALESQP stopping tolerance to $\text{tol}$.
Set bracket $(b_{\text{bot}}, b_{\text{mid}}, b_{\text{top}}) \leftarrow (0, 50, 100)$. Set $b_{\text{pick}} \leftarrow b_{\text{mid}}$. Choose numbers of clusters, $k$, ranking iterations, $n_\text{rank}$, bracketing iterations, $n_\text{bracket}$, and clustering iterations, $n_\text{clust}$. Set $\bp \leftarrow \hat{\bp}$, $L_\text{prev} \leftarrow \infty$ and $\tau \leftarrow 10^{-8}$.
\For{$i=1,2,\dots,n_\text{rank}$} // Ranking phase to find candidates.
    \State Compute targets $\mu \in \mathbb{R}^{N\times8}$ using Algorithm~\ref{algo: clustering}.
    \State Compute weights $\alpha \in \mathbb{R}^{N}$ using formula~\eqref{eq:weights}.
    \State Solve problem~\eqref{eq:optslack} with the relaxed objective function $L_r$ using ALESQP.
    \If{$L_r(\bp) \le \max(\tau, {\varepsilon^2}/{N})$} // Objective tolerance met.
        \State \textbf{break}
    \EndIf
    \If{${|L_r(\bp)-L_\text{prev}|}/{|L_\text{prev}|} \le 10^{-3}$} // Stagnation detected.
        \State \textbf{break}
    \EndIf
    \State Set $L_\text{prev} \leftarrow L_r(\bp)$.
\EndFor

\If{$L_r(\bp) \le \max(\tau, {\varepsilon^2}/{N})$} // Ranking phase sufficient.
        \State Set $b_\text{pick} \leftarrow 99$ and skip bracketing phase.
\EndIf

\State Set $\bp_\text{init} = \bp$, $L_\text{rank} = L_r(\bp)$.

\For{$i=1,2,\dots,n_\text{bracket}$} // Bracketing phase.
  \State Set $\bp \leftarrow \bp_\text{init}$, $L_\text{prev} = L_\text{rank}$.
  \For{$i=1,2,\dots,n_\text{clust}$} // Clustering.
    \State Compute targets $\mu \in \mathbb{R}^{N\times8}$ using Algorithm~\ref{algo: clustering}.
    \State Compute weights $\alpha \in \{0,1\}^{N}$ using formula~\eqref{eq:binary} with $\eta = \lfloor b_\text{mid}/100 \cdot N\rfloor$.
    \State Solve problem~\eqref{eq:optslack} using ALESQP.
    \If{$L(\bp) \le \max(\tau, {\varepsilon^2}/{N})$} // Objective tolerance met.
    \State \textbf{break}
    \EndIf
    \If{${|L(\bp)-L_\text{prev}|}/{|L_\text{prev}|} \le 10^{-3}$} // Stagnation detected.
    \State \textbf{break}
    \EndIf
    \State Set $L_\text{prev} \leftarrow L(\bp)$.
  \EndFor
  \If{$L(\bp) \le \max(\tau, {\varepsilon^2}/{N})$} // Bracketing procedure.
    \If{$b_\text{mid} > b_\text{pick}$}
      $b_\text{pick} \leftarrow b_\text{mid}$
    \EndIf
    \State Set $b_\text{bot} \leftarrow b_\text{mid}$ and
           $b_\text{mid} \leftarrow \lfloor (b_\text{mid} + b_\text{top})/2 \rfloor$.
  \Else
    \State Set $b_\text{top} \leftarrow b_\text{mid}$ and
           $b_\text{mid} \leftarrow \lfloor (b_\text{mid} + b_\text{bot})/2 \rfloor$.
  \EndIf
\EndFor
\State Repeat clustering with $\eta \leftarrow \lfloor b_\text{pick}/100 \cdot N\rfloor$ and $\tau \leftarrow {\varepsilon^2}/{N}$. // Refine result.
\State \textbf{Output:} Mesh $\mathcal{T}_{h,r}$ with points $\bp$.
\end{algorithmic}
\end{algorithm}

\section{Numerical results}
\label{sec:Results}
We now present numerical examples demonstrating the effectiveness of Algorithm~\ref{algo: radapt}.
Each example features ``before'' and ``after'' mesh comparisons,
utilizing Algorithm~\ref{algo: mesh compression} to compute
the mesh compression ratios for a variety of $\varepsilon$ tolerances
within the interval $[10^{-8},1]$.
Misfit plots are included as well, showing how well each cell achieves
the target shapes specified by $\mu$.
Unless stated otherwise, the algorithm parameters are
$\varepsilon = 10^{-10}$, $\text{tol} = 10^{-12}$,
$n_\text{rank} = 6$, $n_\text{bracket} = 8$, and $n_\text{clust} = 200$.
The ALESQP parameters, other than $\text{tol}$, come
directly from~\cite[p.~258]{antil2023alesqp}.

\textbf{Example 1: Structured mesh recovery.}
We start with an example designed to test our algorithm's ability to recover a structured mesh that has been perturbed.
Let $T_h$ be a uniform mesh of square cells on the domain $[-0.5,0.5]^2$ with a random perturbation of strength $0.4\sqrt{N}$ applied to each non-boundary vertex.
The volume constraints are formulated using the global quantities in~\eqref{eq:globalbounds}, with $\gamma=0.4$.
We use the global constraint because it is not clear which $\gamma$ could be used with local constraints~\eqref{eq:localbounds}, as the perturbations are sizeable and such a constraint may prevent the recovery of the structured mesh, which is our verification goal.
We set the number of clusters to $k=2$.
The results are shown in Figure~\ref{fig: example square}.
We observe nearly perfect recovery of the structured mesh, with a compression ratio of almost 100\%.  We recall that due to the storage of at least one cell in the dictionary, the compression ratio can never be exactly 100\%.
The logarithmic plot of cell misfits in Figure~\ref{fig: example square}(f) confirms lossless compression for \emph{all} cells.
For this example, we additionally include a detailed study of the scalability of Algorithm~\ref{algo: radapt}, as we refine the mesh.
Here we use $\varepsilon = 10^{-6}$.
The iteration counts, shown in \ref{table: example square scaling}, demonstrate excellent algorithmic scaling with increasing mesh size.
As most of the computational cost is in the augmented system solves using GMRES, we expect the performance to scale linearly with the number of mesh cells, and to parallelize well, if needed.
We highlight the very low and nearly constant average numbers of GMRES iterations, at approximately three, which is consistent with our discussion in Section~\ref{sec:Algorithms} and the results of~\cite{murphy2000note}.
\begin{figure}[htb!]
  \begin{subfigure}[t]{0.28\textwidth}
    \centering
    \includegraphics[height = 0.9\linewidth]{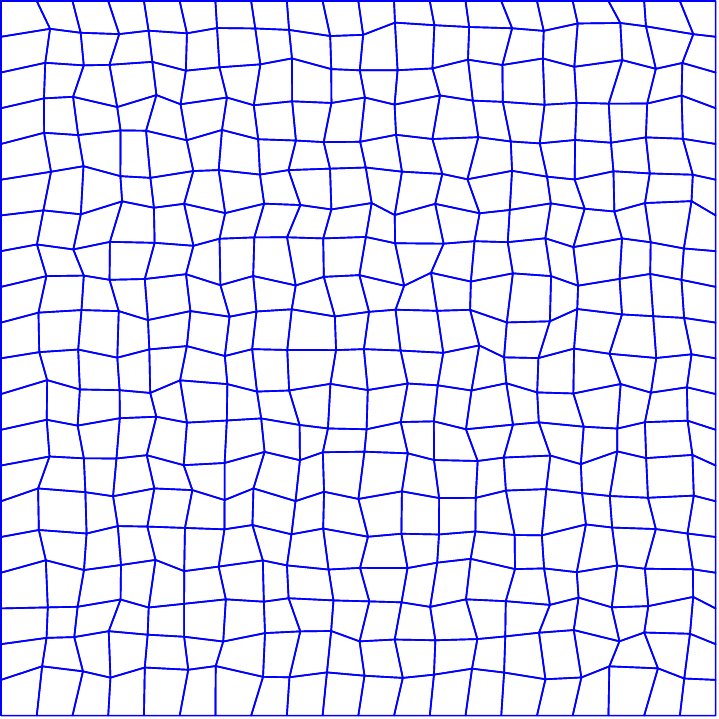}
    \caption{Mesh before.}
  \end{subfigure}
  \hfill
  \begin{subfigure}[t]{0.28\textwidth}
    \centering
    \includegraphics[height = 0.9\linewidth]{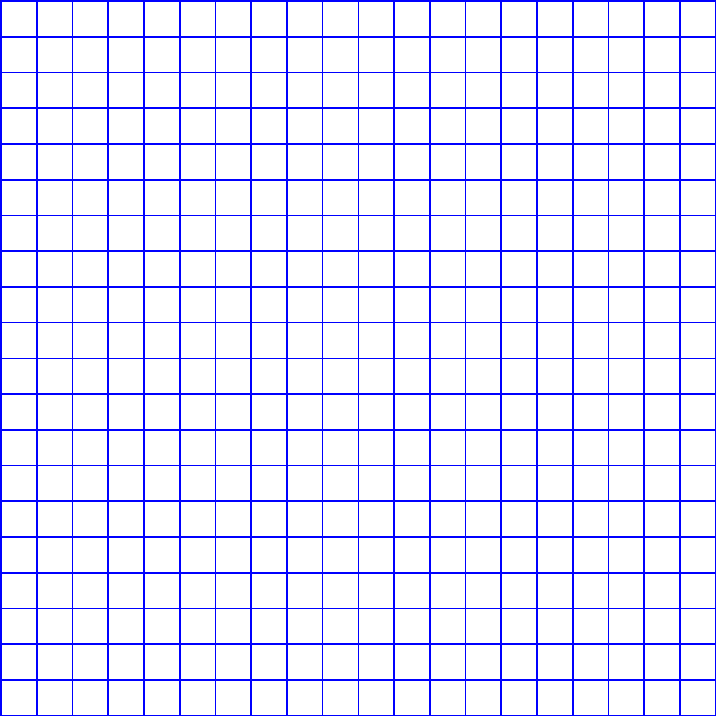}
    \caption{Mesh after.}
  \end{subfigure}
  \hfill
  \begin{subfigure}[t]{0.325\textwidth}
    \centering
    \includegraphics[height = 0.78\linewidth]{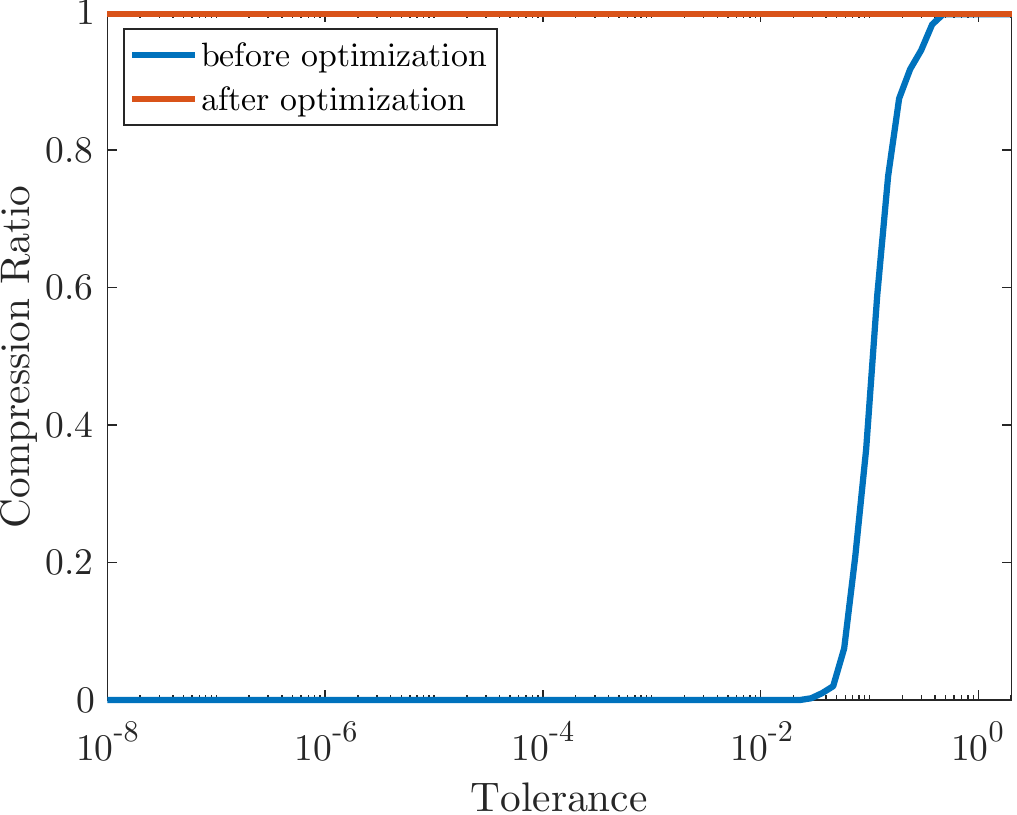}
    \caption{Mesh compression.}
  \end{subfigure}
  \vspace{2ex}
  \begin{subfigure}[t]{0.28\textwidth}
    \centering
    \includegraphics[height = 0.9\linewidth]{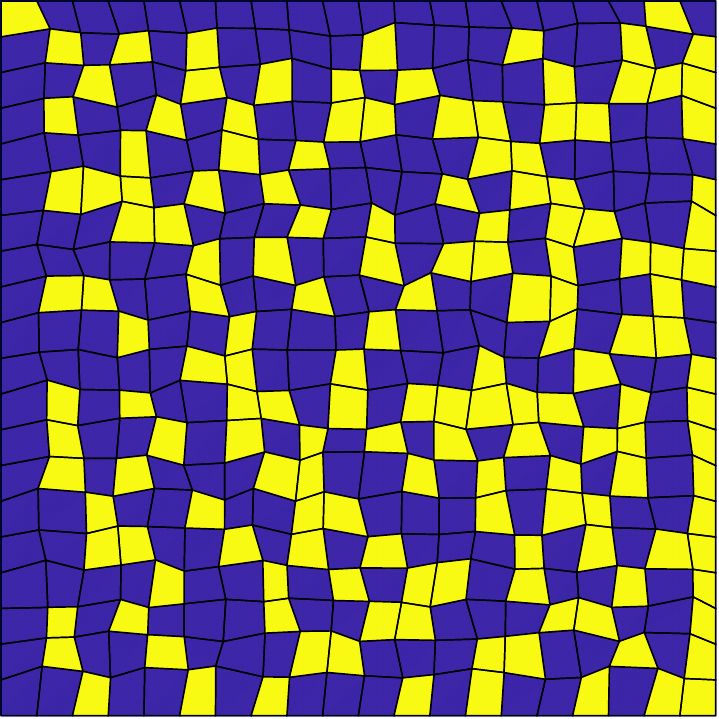}
    \caption{First cluster assignment.}
  \end{subfigure}
  \hfill
  \begin{subfigure}[t]{0.28\textwidth}
    \centering
    \includegraphics[height = 0.9\linewidth]{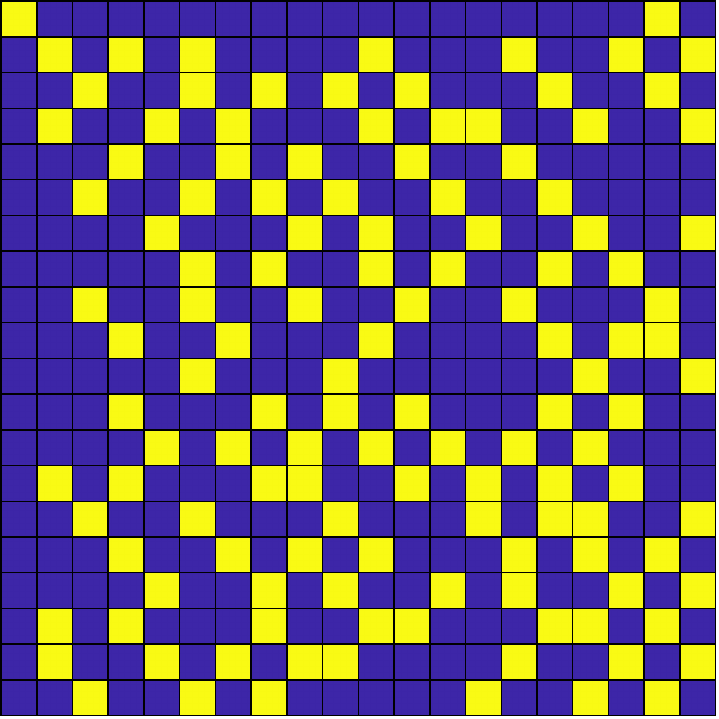}
    \caption{Last cluster assignment.}
  \end{subfigure}
  \hfill
  \begin{subfigure}[t]{0.305\textwidth}
    \centering
    \includegraphics[height = 0.82\linewidth]{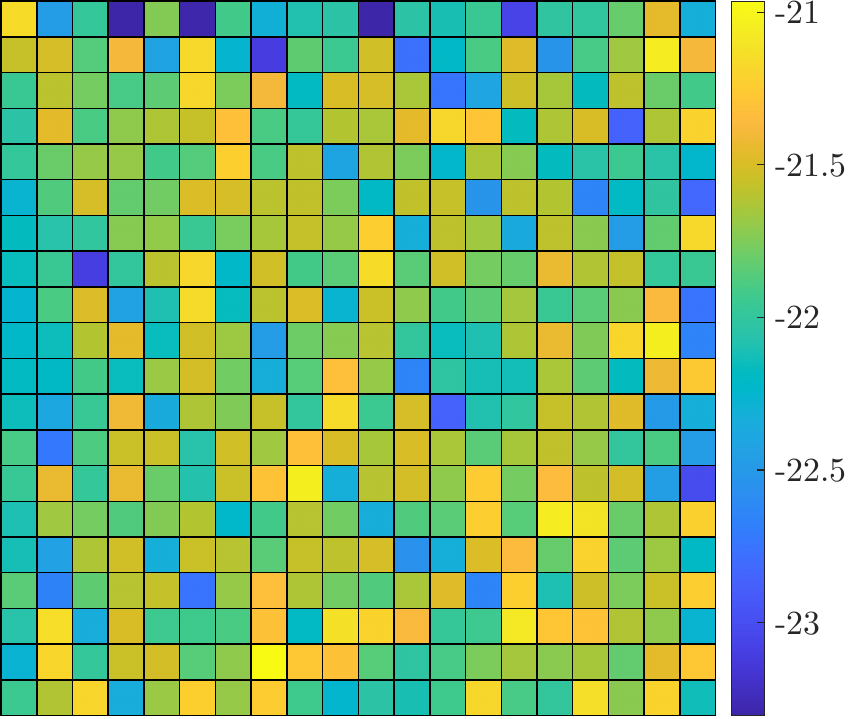}
    \caption{$\log_{10}$ misfit with respect to the cluster medoids.}
  \end{subfigure}  
  \caption{r-adaptive mesh optimization of a randomly perturbed structured $20\times20$ grid. A structured grid is recovered with $10^{-21}$ misfit, losslessly compressing to~99.75\%.}
  \label{fig: example square}
\end{figure}

\begin{table}
\centering
\begin{tabular}{cccccc}
\hline
Mesh & CLS & AL & SQP & CG & Avg.\ GMRES
\\
\hline
$20\times20$ & 7 & 11 & 12 & 223 & 4.51
\\
$40\times40$ & 38 & 71 & 52 & 944 & 3.91
\\
$80\times80$ & 29 & 52 & 39 & 726 & 3.03
\\
$160\times160$ & 28 & 49 & 38 & 851 & 2.87
\\
$320\times320$ & 30 & 54 & 43 & 1522 & 2.91
\\
\hline
\end{tabular}
\caption{Iteration counts for the randomly perturbed square mesh problem with increasing grid sizes.  Here CLS denotes the total number of clustering stages, AL denotes the total number of (outer) augmented Lagrangian iterations in ALESQP, SQP denotes the total number of (inner) SQP iterations, CG denotes the total number of projected CG iterations in the tangential-step QP subproblem, and Avg.\ GMRES denotes the average number of GMRES iterations per call to GMRES. With the exception of the $20\times20$ mesh, all iteration numbers remain in narrow ranges with increasing mesh size.  We note the low average numbers of GMRES iterations, around three, which, combined with other iteration counts, yield a scalable method.}
\label{table: example square scaling}
\end{table}

\textbf{Example 2: Square mesh with poor topology}
Our second example focuses on a mesh with poor topology,
with $10\times10$ boundary division and 99 quadrilaterals in the interior.
This mesh is taken as a two-dimensional slice from the front corner
of the three-dimensional extruded mesh in the previously shown
Figure~\ref{fig: compressible meshes}.
We use two clusters, $k=2$, and we formulate the constraints using the local bounds~\eqref{eq:localbounds}, with $\gamma=0.4$.
The results are shown in Figure~\ref{fig: example square 99}.
This challenging example features 68\% lossless compression, as shown in Figure~\ref{fig: example square 99}(c) and Figure~\ref{fig: example square 99}(f), with the matched cell misfits near $10^{-20}$, again.
A remarkable feature of the results for this example is that our algorithm identified two different clusters containing significant numbers of cells with near-zero misfit with respect to the cluster medoids, exhibiting as two distinct parallelograms. 

\begin{figure}[htb!]
  \begin{subfigure}[t]{0.28\textwidth}
    \centering
    \includegraphics[height = 0.9\linewidth]{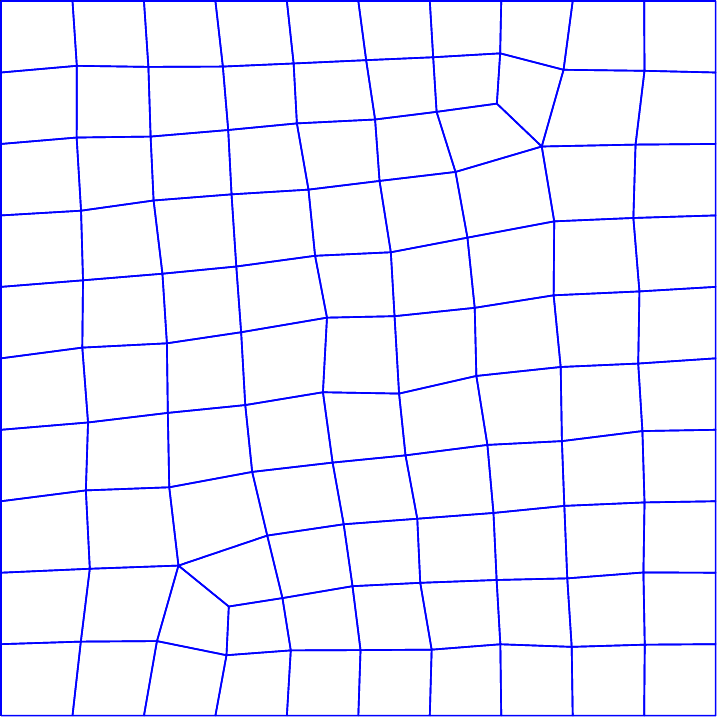}
    \caption{Mesh before.}
  \end{subfigure}
  \hfill
  \begin{subfigure}[t]{0.28\textwidth}
    \centering
    \includegraphics[height = 0.9\linewidth]{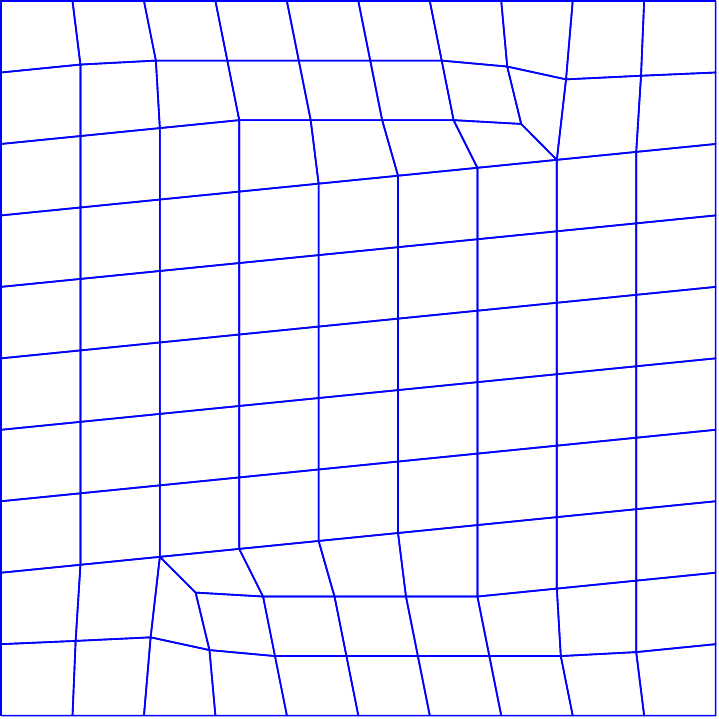}
    \caption{Mesh after.}
  \end{subfigure}
  \hfill
  \begin{subfigure}[t]{0.325\textwidth}
    \centering
    \includegraphics[height = 0.78\linewidth]{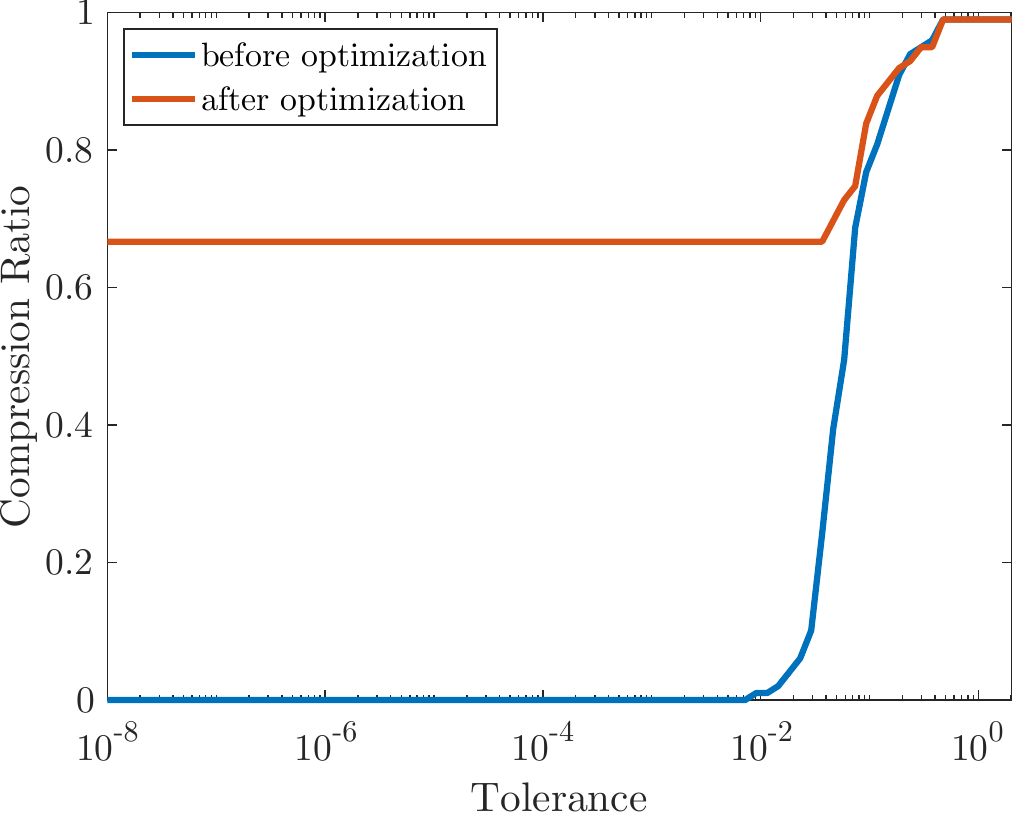}
    \caption{Mesh compression.}
  \end{subfigure}
  \vspace{2ex}
  \begin{subfigure}[t]{0.28\textwidth}
    \centering
    \includegraphics[height = 0.9\linewidth]{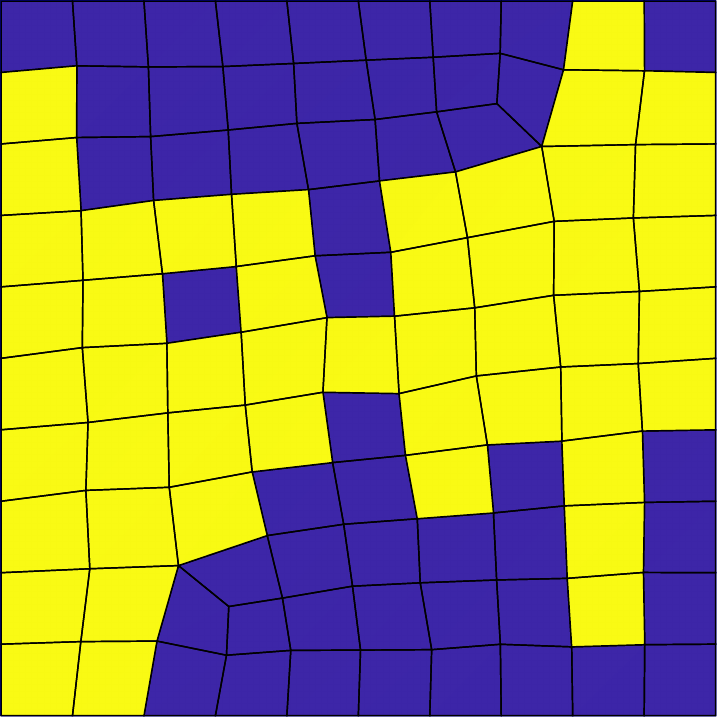}
    \caption{First cluster assignment.}
  \end{subfigure}
  \hfill
  \begin{subfigure}[t]{0.28\textwidth}
    \centering
    \includegraphics[height = 0.9\linewidth]{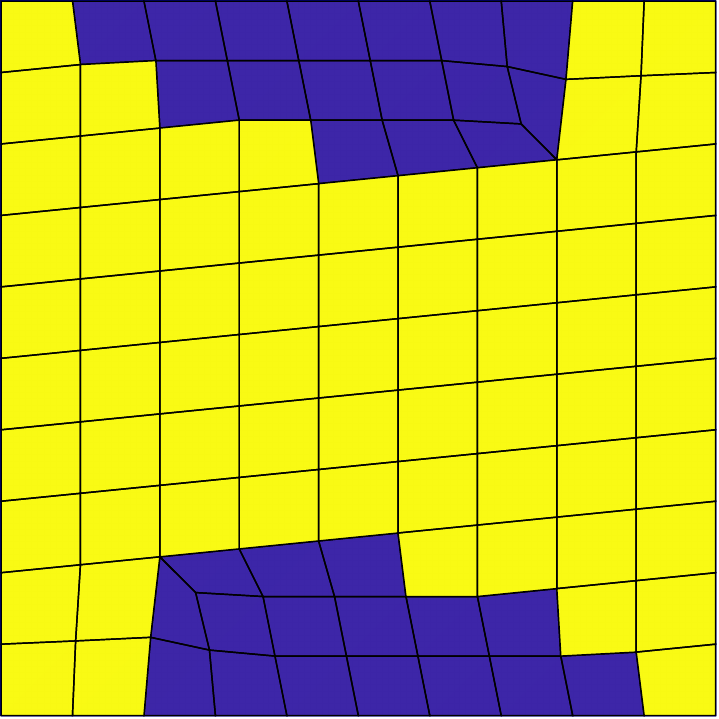}
    \caption{Last cluster assignment.}
  \end{subfigure}
  \hfill
  \begin{subfigure}[t]{0.305\textwidth}
    \centering
    \includegraphics[height = 0.82\linewidth]{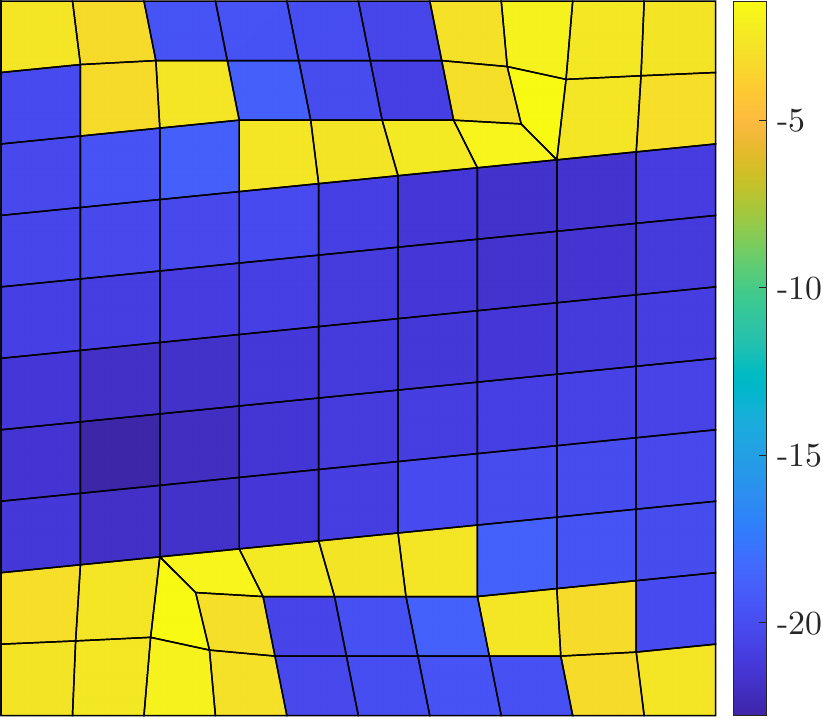}
    \caption{$\log_{10}$ misfit with respect to the cluster medoids.}
  \end{subfigure}  
  \caption{r-adaptive mesh optimization of a mesh with poor topology. Two regions of parallelograms are identified and matched to their clusters' (distinct) cluster medoids with a misfit of near $10^{-20}$, losslessly compressing the mesh to 68\%.}
  \label{fig: example square 99}
\end{figure}

\textbf{Example 3: Circle mesh optimization.}
An interesting question is how one would modify a mesh of the unit circle inside $[-1,1]^2$, which we introduced in Figure~\ref{fig: compression curve}, to increase its compressibility.
The mesh contains 896 cells and 933 vertices, of which 861 nodes are free.
Here we use four clusters, $k=4$, and we formulate the constraints using the local bounds~\eqref{eq:localbounds}, with $\gamma=0.4$.
The results presented in Figure~\ref{fig: example circle} indicate lossless compressibility in the amount of 61\%.
As before, the misfits of the losslessly matched cells are below $10^{-20}$.
While the volume constraints are satisfied to machine precision by the ALESQP algorithm, we note that there are some cell distortions near the corners of the compressible region.
If necessary, these distortions could be prevented by imposing additional constraints on the cell shapes, including, e.g., convexity constraints.
Another noteworthy feature, which contributes to the verification of Algorithm~\ref{algo: radapt}, is that there is at least one zero-misfit cell in each cluster.
This is consistent with the fact that we use a $k$-medoids algorithm to compute the cell targets $\mu$, as medoids are cluster representatives.

\begin{figure}[htb!]
  \begin{subfigure}[t]{0.28\textwidth}
    \centering
    \includegraphics[height = 0.9\linewidth]{figs/circle-before.pdf}
    \caption{Mesh before.}
  \end{subfigure}
  \hfill
  \begin{subfigure}[t]{0.28\textwidth}
    \centering
    \includegraphics[height = 0.9\linewidth]{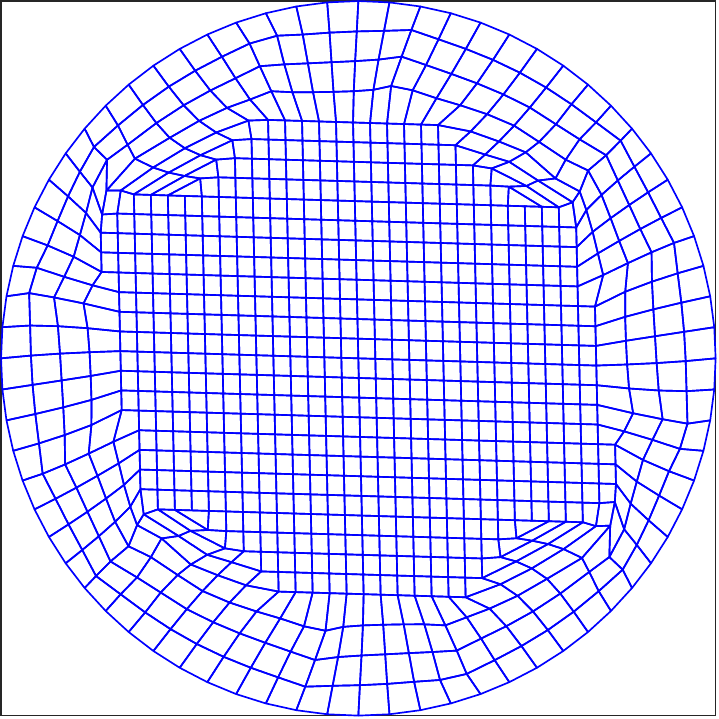}
    \caption{Mesh after.}
  \end{subfigure}
  \hfill
  \begin{subfigure}[t]{0.325\textwidth}
    \centering
    \includegraphics[height = 0.78\linewidth]{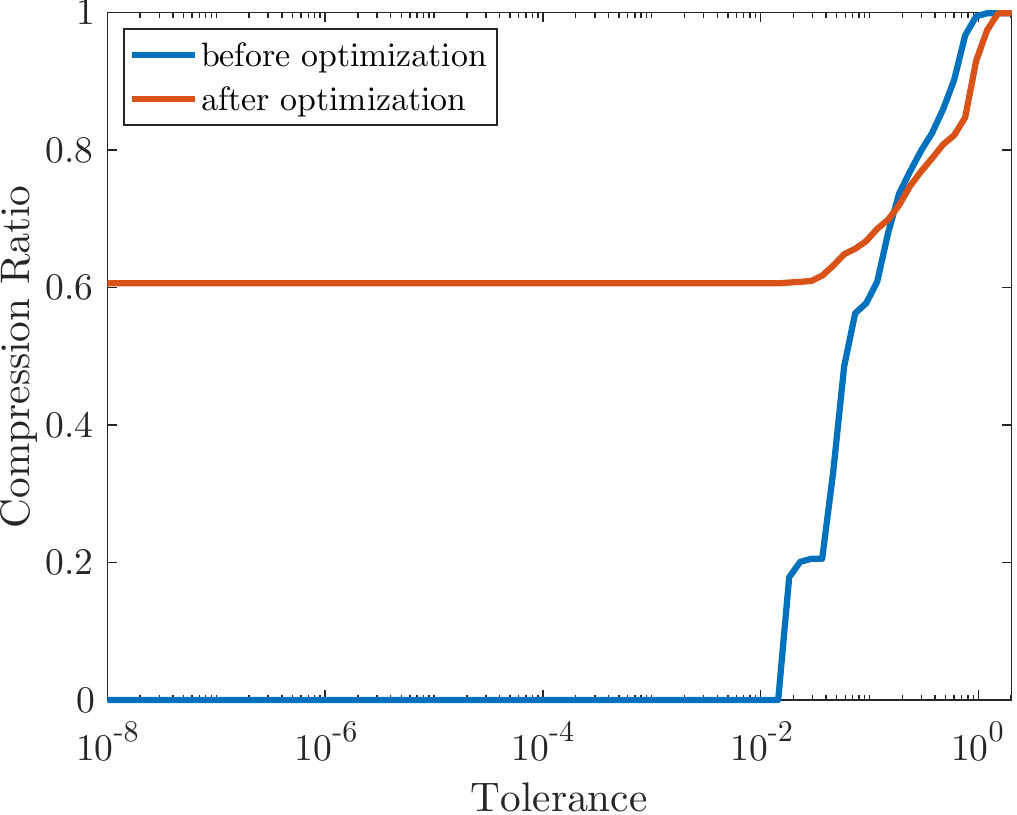}
    \caption{Mesh compression.}
  \end{subfigure}
  \vspace{2ex}
  \begin{subfigure}[t]{0.28\textwidth}
    \centering
    \includegraphics[height = 0.9\linewidth]{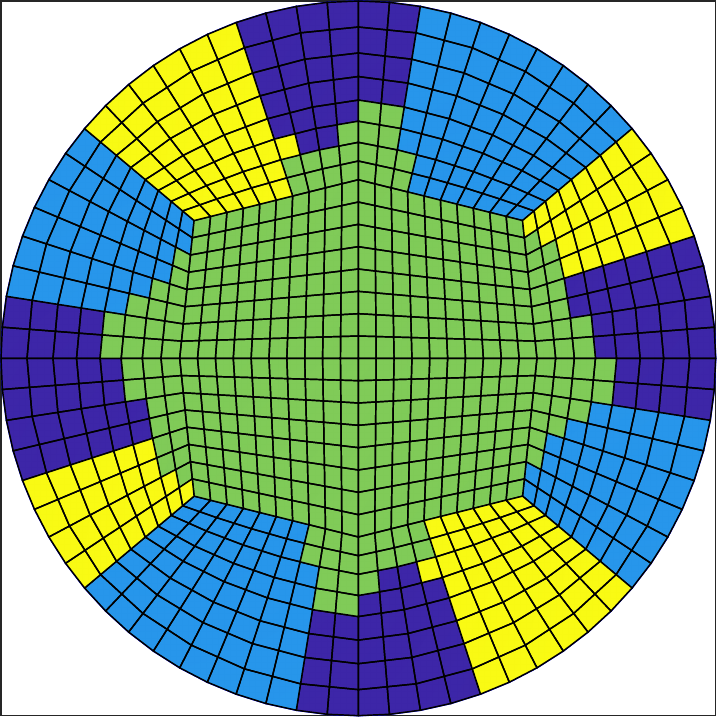}
    \caption{First cluster assignment.}
  \end{subfigure}
  \hfill
  \begin{subfigure}[t]{0.28\textwidth}
    \centering
    \includegraphics[height = 0.9\linewidth]{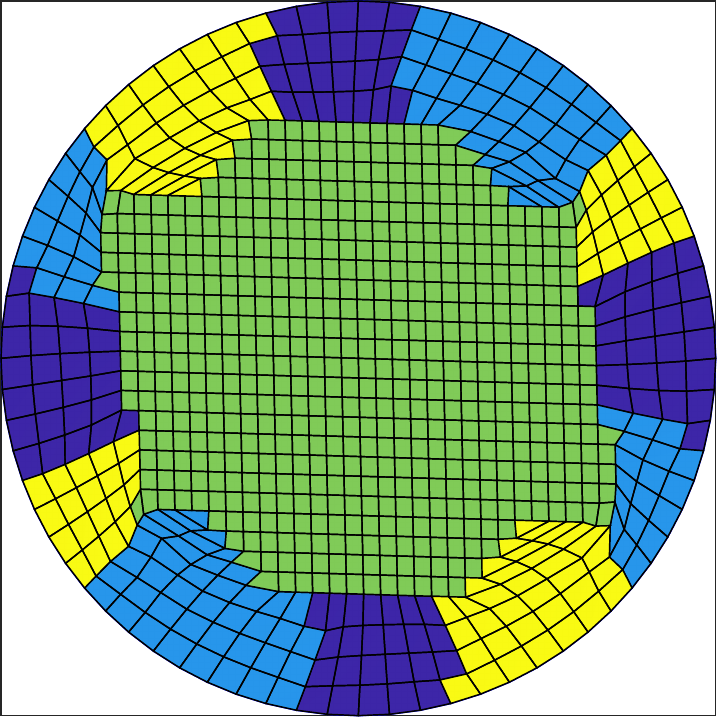}
    \caption{Last cluster assignment.}
  \end{subfigure}
  \hfill
  \begin{subfigure}[t]{0.305\textwidth}
    \centering
    \includegraphics[height = 0.82\linewidth]{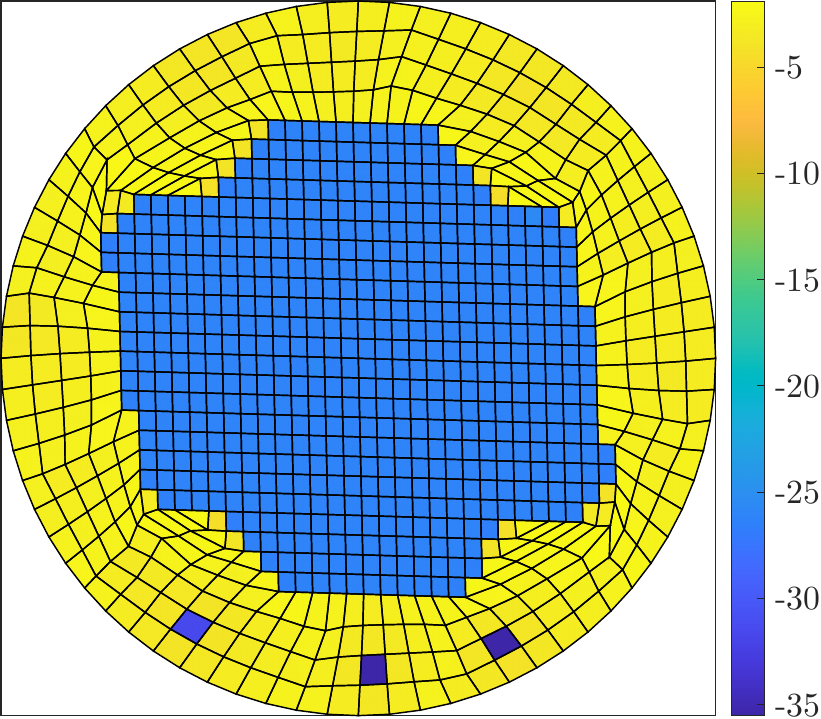}
    \caption{$\log_{10}$ misfit with respect to the cluster medoids.}
  \end{subfigure}  
  \caption{r-adaptive mesh optimization of a circle mesh. The center region is identified as the candidate for near-zero cell misfit and matched to its cluster medoid with a misfit of $10^{-20}$, losslessly compressing the mesh to~61\%.}
  \label{fig: example circle}
\end{figure}
\section{Conclusion}
\label{sec:Conclusion}
We have demonstrated two complementary approaches to reducing the memory burden of large-scale finite element simulations.
Our fist contribution is a dictionary-based compression scheme for finite element meshes containing redundant structure.
Owing to the dictionary, we have shown massive reductions in memory use in the \mrhyde{} finite element simulator.
Our second contribution is an r-adaptive mesh optimization algorithm that combines the recently developed ALESQP method, $k$-medoids clustering, and bracketing, to enhance redundancy in meshes that do not naturally possess it.
Our optimization results demonstrate significant \emph{lossless} compression for a variety of challenging meshes, while maintaining cell shape quality through volume inequality constraints.
Future research directions include extensions to other types of constraints, such as convexity constraints, and algorithmic improvements.
In the large-scale setting, with meshes containing millions or billions of finite elements, the developed algorithms are meant to be applied to subdomains resulting from typical spatial decomposition, rather than the full mesh.
As excellent compression can be shown even for relatively small meshes, an interesting research direction is to further decompose these subdomains into chunks amenable to GPU acceleration and to specialize our dictionary-based compression and optimization algorithms to this setting.

\bibliographystyle{siam}
\bibliography{fem_compression.bib}

\end{document}